\documentclass[11pt]{article}
\usepackage{amsmath,amsthm,amssymb}
\usepackage[T1]{fontenc}
\usepackage[utf8]{inputenc}
\usepackage[dvipdf]{graphicx}
\usepackage{color}
\usepackage{epstopdf}
\usepackage{dsfont}
\usepackage{mathtools}
\usepackage{enumerate}
\usepackage{xcolor}
\usepackage{mathrsfs}
\usepackage[normalem]{ulem}
\usepackage[colorlinks=true,citecolor=red,linkcolor=db,urlcolor=blue,pdfstartview=FitH]{hyperref}
\usepackage[numbers]{natbib}

\allowdisplaybreaks[2]

\providecommand{\keywords}[1]
{
  \small	
  \textbf{\textit{Keywords:}}
}

\definecolor{db}{RGB}{0, 0, 130}
\definecolor{rp}{rgb}{0.25, 0, 0.75}
\definecolor{dg}{rgb}{0, 0.6, 0}

\newtheorem{theorem}{Theorem}[section]

\newtheorem{definition}{Definition}[section]

\newtheorem{assumption}[definition]{Assumption}
\newtheorem{lemma}[definition]{Lemma}

\newtheorem{proposition}[definition]{Proposition}
\newtheorem{remark}[definition]{Remark}

\numberwithin{equation}{section}

\def\pa{\partial}
\def\d{\delta}

\def\a{\alpha}
\def\e{\epsilon}

\def\Ac{\mathcal A}
\def\Jc{\mathcal J}
\def\Bc{\mathcal B}

\def\Fc{\mathcal F}

\def\Hc{\mathcal H}
\def\Lc{\mathcal L}

\def\Pc{\mathcal P}
\def\Sc{{\mathcal S}}

\def\Vc{\mathcal V}

\def\ut{\Tilde{u}}
\def\vt{\Tilde{v}}

\def\E{\mathbb{E}}
\def\F{\mathbb{F}}

\def\N{\mathbb N}

\def\R{\mathbb R}
\def\S{\mathbb S}

\def\vb{\bar{v}}
\def\ub{\bar{u}}
\def\Jb{\bar{J}}

\def\x{\times}
\def\Om{\Omega}

\def \endproof{\hbox{ }\hfill$\Box$}

\title{Comparison for semi-continuous viscosity solutions for second order PDEs on the Wasserstein space}
\author{Erhan Bayraktar
        \footnote{Department of Mathematics, University of Michigan. erhan@umich.edu}
        \and Ibrahim Ekren
        \footnote{Department of Mathematics, University of Michigan. iekren@umich.edu. I. Ekren is partially supported by the NSF grant DMS-2406240}
        \and Xihao He
        \footnote{Department of Mathematics, University of Michigan. hexihao@umich.edu}
        \and Xin Zhang
        \footnote{Department of Finance and Risk Engineering, New York University. xz1662@nyu.edu. X. Zhang is partially supported by the NSF Grant DMS-2508556}
        }
\date{\today}

\begin{document}

\maketitle

\begin{abstract}
In this paper, we prove a comparison result for semi-continuous viscosity solutions of a class of second-order PDEs in the Wasserstein space. This allows us to remove the Lipschitz continuity assumption with respect to the Fourier-Wasserstein distance in \cite{BaEkZh23} and obtain uniqueness by directly working in the Wasserstein space. In terms of its application, we characterize the value function of a stochastic control problem with partial observation as the unique viscosity solution to its corresponding HJB equation. Additionally, we present an application to a prediction problem under partial monitoring, where we establish an upper bound on the limit of regret using our comparison principle for degenerate dynamics.


\end{abstract}


\section{Introduction}
In this paper, we consider a class of second-order PDEs on Wasserstein space
\begin{align}\label{eq:intro}
-\left(\partial_t u +G(\cdot, D_{\mu} u, D_{x\mu} u, \mathcal{H}u) \right)(t,\mu)=0, \quad (t,\mu) \in [0,T) \times \Pc_2(\R^d), 
\end{align}
where $G$ is a generator satisfying suitable continuity conditions, and 
\begin{align*}
    \mathcal{H}u(t,u):= \int D_{x \mu} u(t,\mu)(x) \, \mu(dx) + \int \int D_{\mu\mu}^2 u(t,\mu)(x,y) \, \mu(dx) \mu(dy) \in \mathbb R^{d \times d}
\end{align*}
is the so-called partial Hessian introduced by \cite{BaEkZh23,chow2019partial}. 
Throughout the paper,  $\Pc_2(\mathbb R^d)$ is equipped with the weak topology. We prove a comparison principle for \emph{semi-continuous} viscosity solutions and provide two applications in filtering problems and prediction problems.

PDEs on Wasserstein space appear naturally in filtering problems, where a controller can only obtain access to the conditional law of the state variable due to partial observations \cite{BaEkZh23,martini2023kolmogorov,bandini2019randomized}. These equations have also garnered significant attention in the context of mean-field games and McKean-Vlasov control problems (see, e.g., \cite{cdll2019,MR3752669,MR3753660,MR4499277,MR3739204,cosso2021master}), where the value functions capture interactions within a large population of particles, leading to a state space described by $\mathcal{P}_2(\mathbb{R}^d)$, the space of probability measures with finite second moment.

In \cite{BaEkZh23}, by a doubling variable argument, the authors proved a comparison principle of \eqref{eq:intro} for Lipschitz continuous solutions w.r.t. the Fourier-Wasserstein distance, denoted by $\rho_F$. Since Lipschitz continuity is a metric property and $\rho_F$ is not metrically equivalent to the Wasserstein distance, it is generally quite demanding to verify the Lispchitz property w.r.t. $\rho_F$. To fix this issue, we provide a comparison principle for semi-continuous functions in Theorem~\ref{thm:comparison}, which is more convenient to apply, as it is usually straightforward to verify the continuity w.r.t. the weak topology. Additionally, comparison principles for semi-continuous solutions could be used in the convergence analysis of numerical schemes.

As the first application, we consider a stochastic control problem with partial observation. Under a different set of assumptions from \cite{BaEkZh23}, we characterize the value function as the unique viscosity solution to a PDE of the form~\eqref{eq:intro}. In some sense, the assumption used in the current paper is much weaker: as there is no need to verify the Lipschitz continuity of the value function, we do not assume a lot of regularity of coefficients as in \cite{BaEkZh23}. On the other hand, to weaken the assumption on dynamics drift, we assume that the generator $G$ is uniformly elliptical in terms of $D_{x\mu} u$. Under this condition, we show that the difference in $D_{\mu} u$ terms arising in the doubling variable argument can be controlled by the difference in $D_{x \mu}$ terms \footnote{Mete Soner's Van Eenam Lectures at the University of Michigan were influential in thinking about this cancellation.}, using the so-called commutator estimates inspired by \cite{gozzi2000hamilton}. Note that unlike \cite{gozzi2000hamilton}, we directly work on the Wasserstein space $\Pc_2(\mathbb R^d)$ and we do not need to lift our equation to an $L^2$-space.

We provide another application in prediction problems. It is well-known that solutions to PDEs can be approximated using machine learning-based numerical methods. On the other hand, some learning problems such as expert prediction can be viewed as finite difference algorithms of PDEs; see \cite{2019arXiv190202368B,2020arXiv200308457B,MR4120922,MR4053484,2020arXiv200813703C,2020arXiv200712732D,JMLR:v24:22-1001, MR4253765}. Therefore, the long-time behavior of such learning problems is characterized by certain parabolic PDEs. In \cite{JMLR:v24:22-1001}, the authors interpreted a learning problem as a discrete time zero-sum game with incomplete information. After properly scaling both in time and space variables, the value function converges to a second-order parabolic equation in Wasserstein space of type \eqref{eq:intro} in a formal way; see \cite[Equation (4.11)]{JMLR:v24:22-1001}. Thanks to the comparison principle established in this paper, the limit supremum of discrete time value functions is indeed bounded by the viscosity solution to \cite[Equation (4.11)]{JMLR:v24:22-1001}. We emphasize that for this prediction problem the idiosyncratic noise can be degenerate. However, the generator of the PDE still satisfies the estimates needed for the comparison principle.

The most relevant paper to ours is \cite{daudin2025well}, which proves the existence and a semi-continuous comparison principle for a class of \emph{semi-linear} PDEs on Wasserstein space $\Pc_2(\mathbb T^d)$, i.e., the equation is linear in $D_{x\mu} u$ and $\mathcal{H}u$. Although their assumption on generator $G$ is relative weaker, their comparison cannot be applied to our filtering problem and prediction problem as our unobserved noise term is state dependent and controlled, for which we need the commutator type estimates. Moreover, our equation is defined on $\Pc_2(\mathbb R^d)$, and hence we use the moment penalization technique of \cite{MeYa23} to address the non-compactness issue of $\Pc_2(\mathbb R^d)$. 

Let us discuss other literature. The idea of using the Fourier-Wasserstein distance in the doubling variable argument was brought up in \cite{SoYa24, MeYa23} to show the comparison principle for first-order PDEs in the Wasserstein space. It was then adopted in \cite{BaEkZh23,daudin2025well} for the comparison of second-order PDEs involving the partial Hessian $\mathcal{H}$. Using instead the Wasserstein-$2$ distance as the penalization function in the doubling variable, \cite{2023arXiv230604283B,BeLi24,SaBe24} prove the comparison principles for essentially first-order PDEs. There is another way of proving the comparison for PDEs arising from McKean-Vlasov control problems. Using the particle approximation, \cite{cosso2021master} proved that there are regular functions that are almost solutions up to small errors, and then they proved comparison for first-order equations by a verification argument. Based on similar methods, \cite{BCEQTZ25,2023arXiv231210322C} proved the comparison for other classes of second-order PDEs. 

In \cite{zhou2024viscosity}, Touzi, Zhang, and Zhou tackled the HJB equations on the Wasserstein space by lifting them to the process space, where they introduced a novel concept of viscosity solutions, demonstrating both existence and uniqueness under the assumption of Lipschitz continuity. PDEs in the space of measures also appear in mean-field optimal stopping problems \cite{MR4604196,MR4613226,2023arXiv230709278P}, mean-field control of jump processes \cite{burzoni2020viscosity}, and control problems of occupied processes \cite{STZ24}. Based on the viscosity theory and the comparison principle, the convergence of particle systems in mean field control problems was studied in \cite{2022arXiv221016004T,2022arXiv221100719T,MR4595996}. The convergence rate for PDEs in the Wasserstein space was obtained in \cite{2023arXiv231211373C,2023arXiv230508423D,2022arXiv220314554C,CDJM24,bayraktar2024convergence}. Assuming the existence of smooth solutions to mean-field PDEs, \cite{MR4507678} obtained the optimal convergence rate using a verification argument.

The remainder of the paper is organized as follows. In Section 2, we present notation and definitions of viscosity solutions. In Section 3, we prove a general comparison principle for semi-continuous solutions. In Section 4, we provide two applications in stochastic control with partial observation and prediction problems, respectively. In Section 5, we first give the commutator estimates and use these to prove the results in Section 4.

\section{Preliminaries}
    
    \subsection{Notations}
    Let $T$ be the finite horizon with $0 < T < +\infty$.
    For any measurable space $(\Om,\mathfrak{F})$,
    let $\Pc(\Om)$ denote the collection of all probability measures on $(\Om,\mathfrak{F})$. We take a positive constant $\lambda> \lfloor d/2 \rfloor +2$. 
 For any $\rho \in \mathcal{P}_2(\mathbb R^d)$ and any bounded continuous function $f:\mathbb R^d \to \mathbb R$, we denote $\langle f,\rho\rangle := \int f(x) \, \rho(dx)$.   
 The Fourier Wasserstein $2$-distance $\rho_F$ between two probability measures $\rho_1$ and $\rho_2$ with $\rho_1, \rho_2 \in \Pc_2(\R^d)$ is defined as
    \begin{equation*}
        \rho_F^2(\rho_1,\rho_2)
        ~ := ~
        \int_{\R^d}\frac{|F_k(\rho_1-\rho_2)|^2}{(1 + |k|^2)^\lambda}dk,
    \end{equation*}
    where for $k \in \R^d$, $f_k(x) := (2\pi)^{-\frac{d}{2}}e^{i k \cdot x}$, and $F_k(\rho) := \langle f_k, \rho \rangle$. For any complex number $z$, we denote its conjugate by $z^*$. For $\theta_1=(t_1,\mu_1,m_1), \, \theta_2=(t_2,\mu_2,m_2) \in \Theta:=[0,T] \times \Pc_2(\R^d) \x \R^d$, define 
    $$
        d_F^2(\theta_1,\theta_2)=|t_1-t_2|^2+|m_1-m_2|^2+\rho_F^2(\mu_1,\mu_2).
    $$
    
    \vspace{0.5em}

    Now denote by $\S_d$ the space of all $d \x d$-dimensional matrices with real entries, equipped with the Frobenius Norm $|\cdot|$. Let us denote by $B^d_l$ ($B^{d\times d}_l$) the set of Borel measurable functions on $\R^d$ taking values in $\R^d$ ($\R^{d\times d}$) with at most linear growth, and define $|f|_l=\sup_{x \in {\mathbb R^d}} \frac{|f(x)|}{1+|x|}$ for any $f \in B^d_l, B^{d \x d}_l$.

    \begin{remark}
        $\mathrm{(i)}$
        According to \cite{BaEkZh23,MeYa23} and the references therein, the Fourier-Wasserstein distance $\rho_F$ metrizes  the weak topology on $\mathcal{P}(\mathbb R^d)$. Moreover, $\rho_F$ is bounded by the Wasserstein-$1$ distance on $\mathcal{P}_1(\mathbb R^d)$, i.e., $\rho_F \leq C W_1$ for some positive constant $C$. 

      \noindent $\mathrm{(ii)}$
        The space $(\Pc(\R^d),\rho_F)$ is not complete, and we have the following counterexample:
        Take $\rho$ to be the density of the $d$-dimensional normal distribution, and set
        \[
        \rho_n(x)=n^{-d}\rho(x/n),\quad
        \mu_n=\rho_n(x)\,dx,
        \quad
        \mu:= \rho(x)\,dx.
        \]
        Then $F_k(\mu_n)=F_{nk}(\mu)=e^{-|nk|^2/2}$. Hence
        \begin{align*}
        \rho^2_F(\mu_n,0)
        =\int_{\mathbb{R}^d} \frac{|F_k(\mu_n)|^2}{(1+|k|^2)^{\lambda}}\,\,dk 
        = \int_{\mathbb{R}^d}\!\frac{e^{-|nk|^2}}{(1+|k|^2)^{\lambda}}dk
        \xrightarrow{n \longrightarrow +\infty} 0. 
        \end{align*}
        So $\mu_n\to 0$ in $H^{-s}$, but $0$ is not a probability measure. Thus
        $\mathcal P(\mathbb{R}^d)$ is not $\rho_F$-complete.
        
        \end{remark}

    \subsection{Derivatives on the space of probability measures}
    
    \begin{definition}
        \begin{enumerate}[(i)]
            \item A function $u: \Pc_2(\R^d) \longrightarrow \R$ is said to have a linear functional derivative if there exists
                \begin{equation*}
                    \delta_{\mu}u: \Pc_2(\R^d) \x \R^d \longrightarrow \R
                \end{equation*}
            such that $\delta_\mu u$ is continuous in the product topology and
                \begin{itemize}
                    \item for each $\mu \in \Pc_2(\R^d)$, the mapping $x \longmapsto \delta_{\mu} u(\mu,x) \in B_l$, 
                    \item for all $m_1, m_2 \in \Pc_2(\R^d)$,
                        \begin{equation}
                            u(m_1) - u(m_2) 
                            ~ = ~
                            \int_0^1 \int_{\R^d}
                                \delta_{\mu} u(\lambda m_1 + (1 - \lambda)m_2,x)(m_1 - m_2)(dx) d\lambda.
                        \end{equation}
                \end{itemize}
            \item A function $u: \Pc_2(\R^d) \longrightarrow \R$ is said to have a second-order linear functional derivative, if for any $x \in \mathbb{R}^d$, $\mu \mapsto \delta_{\mu}u$ has a linear functional derivative, denoted by $\delta_{\mu}^2 u (\mu)(x,y)$, and $(\mu,x,y) \mapsto \delta^2_{\mu} u(\mu,x,y)$ is continuous in the product topology. 
            \item A function $u: \mathcal{P}_2(\mathbb R^d) \longrightarrow \R$ is said to be second-order Lions differentiable, if $\mathcal{P}_2(\mathbb R^d) \times \mathbb R^d \times \mathbb R^d \ni (\mu,x,y) \to (\pa_x\delta_\mu u, \pa^2_{xx}\delta_\mu u, \pa_x\pa_y\delta^2_\mu u)(\mu,x,y) \in B^d_l \x B_l^{d \x d} \x B_l^{d \x d}$ exists and is continuous in all its variables. We denote $ D_\mu u := \pa_x\delta_\mu u$,$D_{x\mu} u := \pa^2_{xx}\delta_\mu u$, and $D^2_{\mu\mu} u:= \pa_x\pa_y\delta^2_\mu u$.
        \end{enumerate}
    \end{definition}
    For any second-order Lions differentiable function $\mu \longmapsto u(\mu)$, we define the partial Hessian
    \begin{align*}
        \Hc u :=& \pa^2_{zz}u((I_d + z)_\sharp \mu)|_{z = 0}
        =
        \displaystyle
            \int_{\R^d}\int_{\R^d}D^2_{\mu\mu}u(\mu,x,y)\mu(dx)\mu(dy) 
                + \int_{\R^d} D_{x\mu} u(\mu,x)\mu(dx).
    \end{align*}
    
    \subsection{The non-linear PDE and its viscosity solution}

        Our aim is to obtain the comparison principle for viscosity solutions of the PDE 
        \begin{align}\label{eq:HJB}
            -\left(\pa_t u+G(\cdot, D_\mu u,D_{x\mu}u,\Hc u)\right)(t,\mu)=0,
            ~(t,\mu) \in [0,T) \x \Pc_2(\R^d),
        \end{align}
        where $G$ is a function from $[0,T] \x \Pc_2(\R^d) \x B^d_l \x B^{d\x d}_{l} \x \S_d$ to $\R$.
        As $\mathcal{H}u$ is the second-order derivative in the barycenter of measures, if $u$ is a classical solution to \eqref{eq:HJB}, then we can verify that $\ub(t,\mu,m) := u(t, (I_d + m)_\sharp\mu)$ satisfies the new PDE
        \begin{align}\label{eq:NewHJB}
            -(\pa_t \ub + G^e(\cdot,D_\mu\ub,D_{x\mu}\ub,\pa^2_{mm}\ub))(t,\mu,m) ~ = ~ 0,
            ~(t,\mu, m) \in [0,T) \x \Pc_2(\R^d) \x \R^d,
        \end{align}
        where for any $(t,\mu,m, p,q, X) \in   [0,T]  \x \Pc_2(\R^d) \x \R^d \x B_l^d \x B_l^{d \x d} \x \S_{d}$,
        $$
            G^e(t,\mu, m,p,q,X) := G(t,(I_d + m)_\sharp\mu,p(\cdot-m),q(\cdot - m),X).
        $$
        The partial Hessian term $\mathcal{H}u$ in \eqref{eq:HJB} is turned into a finite dimensional second-order derivative $\partial_{mm}$ in \eqref{eq:NewHJB}. For the convenience to apply Ishii's lemma, in the rest of the paper, we consider the equation \eqref{eq:NewHJB} instead of \eqref{eq:HJB}. 
     
    \begin{remark}
    In fact, we can consider a more general case where $u$ and $G$ depend on a variable $y \in \mathbb{R}^d$, and $G$ additionally depends on $u$. Then a comparison principle can be proved using almost the same techniques and arguments as in this paper. 
    \end{remark}
    
        We introduce its second-order jets and viscosity solution as follows:
    \begin{definition}
         We say a function $\ub:\Theta=[0,T] \times \Pc_2(\R^d) \x \R^d \longrightarrow \R$ is partial $C^2$-regular
            if, for any $\mu \in \mathcal{P}_2(\mathbb R^d)$,
            $(t, m) \in \Theta \longmapsto \ub(t,\mu, m)$ is $C^{1,2}$, $(t,\mu,m) \mapsto \ub(t,\mu,m)$ is continuous, and $(D_\mu \ub, D_{x\mu} \ub) \in B^d_l \x B_l^{d \x d}$ exists and is continuous in all its variables.
    \end{definition}

    \begin{definition}
        Let $\ub:\Theta \longmapsto \R$ be a locally bounded function, and $\theta \in [0,T) \x \Pc_2(\R^d) \x \R^d $.
        We define the (partial) second-order superjet 
        $J^{2,+}\ub(\theta) \subset \R  \x B_l^d \x B_l^{d \x d} \x \S_{d}$ of $\ub$ at $\theta$ by
        \begin{align*}
            J^{2,+}\ub(\theta) 
            := 
            \{(\pa_t\phi,  D_\mu \phi, D_{x\mu}\phi, \pa^2_{mm}\phi)(\theta): ~
            & \ub - \phi~\mbox{has a local maximum at}~\theta, \\ 
            & \phi ~\mbox{is partial}~C^2\mbox{-regular}\},
        \end{align*}
        the second-order subjet $J^{2,-}\ub(\theta) := -J^{2,+}(-\ub)(\theta)$, and the closure of the jets as
        \begin{align*}
            \Jb^{2,+}\ub(\theta) 
            := &
            \{(b, p, q, X) \in \R  \x B_l^d \x B_l^{d \x d} \x \S_{d}: \\
            & \ub - \phi^n~\mbox{has a local maximum at}~\theta_n,~
             \phi^n~\mbox{is partial}~C^2\mbox{-regular}; ~\mbox{as}~n \longrightarrow +\infty, \\
             & (\pa_t\phi^n,D_\mu \phi^n, D_{x\mu}\phi^n, \pa^2_{mm}\phi^n)(\theta_n) \longrightarrow (b, p, q, X), 
             ~\theta_n \longrightarrow  (\theta), 
             \ub(\theta_n) \longrightarrow  \ub(\theta) \},
        \end{align*}
        and $\Jb^{2,-}\ub(\theta) := - \Jb^{2,+}(-\ub)(\theta) $.
    \end{definition}
    
    \begin{definition}
        We say that $u$ is a viscosity subsolution (supersolution) to \eqref{eq:HJB} if 
             \begin{align*}
                   &\bar u: [0,T]  \times \Pc_2(\R^d) \times \R^d \ni (t, \mu, m) \mapsto u(t , (I_d+m)_{\#}\mu),
             \end{align*}
        is a locally bounded function satisfying that for $(t,\mu, m) \in [0,T) \x \Pc_2(\R^d) \x \R^d$,
        \begin{equation*}
            - b - G^e(t,\mu,m, p, q, X) \le (\ge) ~0, 
            ~\mbox{for any}~
            (b, p, q, X) \in J^{2,+}\ub(t,\mu, m)~ (J^{2,-}\ub(t,\mu, m)).
        \end{equation*}
    \end{definition}

\section{Ishii's Lemma and the Comparison principle}

    We equip the space $[0,T] \x \Pc_2(\R^d) \x \R^d $ with the product topology, where $\Pc_2(\R^d)$ is equipped with the weak convergence topology. Define an auxiliary function 
    \begin{align}\label{eq:vartheta}
        \vartheta: [0,T] \x \Pc_2(\R^d) \x \R^d \ni \theta = (t,\mu,m) \longmapsto  |m|^2~+\int_{\R^d} |x|^2 \, \mu(dx). 
    \end{align}    

  \begin{lemma}[Ishii's Lemma]\label{lemm:ishii}
    Suppose that $u, -v: [0,T] \x  \Pc_2({\R^d}) \rightarrow \R$ are bounded upper-semicontinuous functions.   
    For any $\delta>0$, introduce 
    \begin{align*}
        &\tilde u: [0,T] \times \Pc_2(\R^d) \times \R^d  \ni (t, \mu, m) \mapsto u(t, (I_d+m)_{\#}\mu)- \delta \vartheta(t, \mu, m), \\
        &\tilde v: [0,T] \times \Pc_2(\R^d) \times \R^d  \ni (s, \nu, n) \mapsto v(s, (I_d+n)_{\#}\nu)+ \delta \vartheta(s, \nu, n).
    \end{align*}
    Then there exists  a local maximum $ (\theta^*,\iota^*)=((t^*,\mu^*,m^*),(s^*,\nu^*,n^*))$ of 
            \begin{align}\label{eq:ishiimax}
             (\theta,\iota) \longmapsto \tilde u(\theta)- \tilde v(\iota) -\frac{1}{2\e} \left(|t-s|^2+|m-n|^2+\rho_F^2(\mu,\nu) \right) .
             \end{align}
            Assume $ \theta^*, \iota^*$ are in the interior $[0,T) \times \Pc_2(\R^d) \x \R^d$, and denote
        \begin{align}
            \Lc(\eta,\mu,\nu):=&2\int \frac{Re(F_k(\eta)(F_k(\mu)-F_k(\nu))^*)}{(1+|k|^2)^\lambda}dk, \label{eq:mathcalL}\\
\Phi(\mu):=&2\rho^2_F(\mu,\mu^*)+\Lc(\mu,\mu^*,\nu^*)\label{eq:defpsi},\\
            \Psi(\nu):=& 2\rho^2_F(\nu,\nu^*)-\Lc(\nu,\mu^*,\nu^*),\notag
        \end{align}
        where $F_k^*(\mu), F_k^*(\nu)$ are the conjugates of $F_k(\mu), F_k(\nu)$, respectively.
     Then for any $\a>0$, there exist $X^*,Y^*
    \in \mathbb{S}_{d}$ such that 
    \begin{align*}
     & \left(\frac{1}{\e}(t^*-s^*), 
        \frac{D_\mu \Phi(\mu^*)(\cdot)}{2\e},
        \frac{D_{x\mu} \Phi(\mu^*)(\cdot)}{2\e},X^* \right) 
        \in \bar J^{2,+} \tilde u(\theta^*), \\
    & \left(\frac{1}{\e}(t^*-s^*), 
        -\frac{D_\mu \Psi(\nu^*)(\cdot)}{2\e},
        -\frac{D_{x\mu} \Psi(\nu^*)(\cdot)}{2\e},-Y^* \right) 
        \in \bar J^{2,-} \tilde v (\iota^*),
x        \end{align*}
        as well as 
    \begin{align*}
        -\left(\frac{1}{\alpha}+\frac{2}{\e}\right)I_{2d}\leq 
        &\begin{pmatrix}
            X^*&0\\0& Y^*
        \end{pmatrix}
        \leq 
        \left(\frac{1}{\e}+\frac{2\alpha}{\e^2} \right)
        \begin{pmatrix}
            I_{d}&-I_{d}\\-I_{d}&I_{d}
        \end{pmatrix}.
    \end{align*}
      
    \end{lemma}

    \begin{proof}
        Note that $(\mu, m) \mapsto (id+m)_{\#}\mu$ is continuous and $ \mu \mapsto \int_{\R^d}|x|^2 \, d \mu$ is lower semicontinuous with respect to the weak topology. Therefore, $\tilde u, - \tilde v$ are still upper-semicontinuous functions. Indeed for any sequence $(m_k,\mu_k) \to (m,\mu)$ and a bounded Lipschitz function $f \in C_b(\R^d; \R)$, 
        \begin{align*}
          & \left|  \lim\limits_{k \to \infty } \int f(x) \, (I_d+m_k)_{\#}\mu_k(dx)- \int f(x) \, (I_d+m)_{\#}\mu(dx) \right| \\
            & \leq Lip(f) \limsup_{k\to \infty} |m_k-m| +  \limsup_{k \to \infty }  \left| \int f(x+m) \, (\mu_k-\mu)(dx) \right|=0,
        \end{align*}
        where  $Lip(f)$ denotes the Lipschitz coefficient of $f$. Moreover for each $N>0$, 
        \begin{align*}
        \liminf_{k \to \infty} \int |x|^2 \, d \mu_k \geq \liminf_{k \to \infty} \int |x|^2\wedge N  \, d \mu_k = \int |x|^2 \wedge N \, d\mu,
        \end{align*}
        and hence $\liminf_{k \to \infty} \int |x|^2 \, d \mu_k \geq  \int |x|^2 \, d \mu$.
        
        Denote the maximum value of \eqref{eq:ishiimax} obtained by $a \in \R$. We consider the set $A$ of $(\theta,\iota)$ such that 
        \begin{align*}
             \tilde u(\theta)- \tilde v(\iota) -\frac{1}{2\e} \left(|t-s|^2+|m-n|^2+\rho_F^2(\mu,\nu) \right) >a-1,
        \end{align*}
        and hence 
        \begin{align*}
            \delta \left( \vartheta(\theta)+\vartheta(\iota) \right) \leq 1-a +\lVert u \rVert_{\infty}+\lVert v \rVert_{\infty}.
        \end{align*}
        Therefore the closure of $A$ is compact, and we denote by $(\theta^*,\iota^*)=((t^*,\mu^*,m^*),(s^*,\nu^*,n^*))$ a global maximum of \eqref{eq:ishiimax}. We assume that $t^*, s^* <T$.

        \medskip

        \noindent {\it Step 1: Reducing to a finite dimensional problem.}
        Define the functions
                \begin{align*}
                 u_1(t,\mu,m)&:= \tilde u(t,\mu,m) -\frac{1}{2\e}\Phi(\mu), \\
                 v_1(s,\nu,n)&:=\tilde v(s,\nu,n)+\frac{1}{2\e} \Psi(\nu),
            \end{align*}
            where we recall $\Psi$ and $\Phi$ are defined in equations \eqref{eq:mathcalL} and \eqref{eq:defpsi}.
            Given any Hilbert space $(L,\langle \cdot \rangle )$ and $a,b,c,d \in L$, it is straightforward that 
            \begin{align*}
            |a-b|^2+|c-d|^2-2\langle a-b,c-d\rangle = |a-b-(c-d)|^2 \leq 2 (|a-c|^2+|b-d|^2).
            \end{align*}
            Therefore we have 
            \begin{align}
               &|F_k(\mu)-F_k(\nu)|^2+|F_k(\mu^*)-F_k(\nu^*)|^2\notag\\
               &\quad\quad\leq2(|F_k(\mu)-F_k(\mu^*)|^2+|F_k(\nu)-F_k(\nu^*)|^2)\notag\\
               &  \quad \quad\quad+ 2Re((F_k(\mu)-F_k(\nu))(F_k(\mu^*)-F_k(\nu^*))^*). \label{ineq:1}
               \end{align}
         Hence it holds that 
        \begin{align*}
            2\rho^2_F(\mu,\mu^*)+2\rho^2_F(\nu,\nu^*)+\Lc(\mu,\mu^*,\nu^*)-\Lc(\nu,\mu^*,\nu^*)\geq \rho^2_F(\mu,\nu)+\rho^2_F(\mu^*,\nu^*).
        \end{align*}
        Thus, we have the inequality for any $(\theta,\iota) = ((t,\mu,m),(s,\nu,n))$
        \begin{align*}
             & u_1(t,\mu,m)-v_1(s,\nu,n)-\frac{1}{2\e}\left(|t-s|^2+|m-n|^2\right)\\
             &= \tilde u(t,\mu,m)-\tilde v(s,\nu,n)-\frac{1}{2\e}\left(|t-s|^2+|m-n|^2\right)\\
             &\quad -\frac{ 1}{2\e}\left(2\rho_F^2(\mu,\mu^*)+2\rho^2_F(\nu,\nu^*)\right)
                -\frac{1}{2\e}\left(\Lc(\mu,\mu^*,\nu^*)-\Lc(\nu,\mu^*,\nu^*)\right)\\
             &\leq \tilde u(t,\mu,m)-\tilde v(s,\nu,n)-\frac{1}{2\e}\left(|t-s|^2+|m-n|^2\right)
                -\frac{1}{2\e}(\rho^2_F(\mu,\nu)+\rho^2_F(\mu^*,\nu^*))\\
             &=\tilde u(\theta)-\tilde v(\iota)-\frac{1}{2\e} d_F^2(\theta,\iota)-\frac{1}{2\e}\rho^2_F(\mu^*,\nu^*).
        \end{align*}
        
        Furthermore, for $(\mu,\nu)=(\mu^*,\nu^*)$, the inequality \eqref{ineq:1} is an equality. 
        Thus, the function
        \begin{align}\label{eq:optu1v1}
        u_1(t,\mu, m)-v_1(s,\nu, n)-\frac{1}{2\e}\left(|t-s|^2+|m-n|^2\right)
        \end{align}
        admits the strict global maximum at $((t^*,\mu^*,m^*),(s^*,\nu^*,n^*))$.
        Define the mappings 
        \begin{align}
            [0,T]\times \R^d \ni  (t,m) \mapsto  u_2(t,m)&:=\limsup_{( \tilde t, \tilde m) \to (t,m)}\sup_{\mu\in \Pc_{2}(\R^d)} u_1(\tilde t, \mu,\tilde m) \label{eq:defu2},\\
            [0,T]\times \R^d \ni (s,n) \mapsto v_2(s,n)  &:=\liminf_{( \tilde s, \tilde n) \to (t,m)}\inf_{\nu \in \Pc_{2}(\R^d)} v_1(\tilde s,\nu, \tilde n).\label{eq:defv2} 
        \end{align}
        It can be easily seen that $u_2, -v_2$ are upper-semicontinuous. By the maximality of $ u_1(t,\mu, m)-v_1(s,\nu, n)-\frac{1}{2\e}\left(|m-n|^2+|t-s|^2\right)$ at $(\theta^*,\iota^*)$, the function 
        \begin{align}\label{eq:optfinite}
              u_2(t, m)-v_2(s, n)-\frac{1}{2\e}\left(|t-s|^2+|m-n|^2\right)
        \end{align}
         admits a global maximum at $((t^*,m^*),(s^*,n^*))$. 

         \medskip
         \noindent {\it Step 2: Applying finite dimensional Ishii's lemma.}
         By Ishii's Lemma in \cite[Theorem 3.2]{usersguide} applied to \eqref{eq:optfinite}, for $\alpha>0$,
        there exists a sequence indexed by $k$
        $$(t^*_k,s^*_k,m^*_k,n^*_k,h^*_k,j^*_k,X_k^*,Y_k^*)\in (\R)^2\times (\R^d)^2\times \R^2\times  (\S_{d})^2$$ 
        so that 
        \begin{align}\label{maxu2v2}
            u_2(t,m)-v_{2}(s,n)-h^*_k(t-t^*_k)-j^*_k(s-s^*_k)
            &-\frac{1}{2}(m-m^*_k)^\top X^*_k(m-m^*_k)
            \\ &    -\frac{1}{2}(n-n^*_k)^\top Y^*_k(n-n^*_k)\notag
        \end{align}
        admits a local maximum at $((t^*_k,m^*_k),(s^*_k,n^*_k))$, and we have the following bounds and limits as $k \to \infty$,
        \begin{align}\label{eq:tstarn}
            ((t^*_k,m^*_k),(s^*_k,n^*_k)) &\longrightarrow\left((t^*,m^*),(s^*,n^*)\right),\\
            (u_2(t^*_k,m^*_k),v_{2}(s^*_k,n^*_k))&\longrightarrow (u_2(t^*,m^*),v_{2}(s^*,n^*)),\notag\\
            (h^*_k,j^*_k,X^*_k,Y^*_k) &\longrightarrow\left(\frac{1}{\e}(t^*-s^*),-\frac{1}{\e}(t^*-s^*),X^*,Y^*\right),  \notag
        \end{align}
        as well as the inequality 
        \begin{align}
        \label{eq:orderXn}
            -\left(\frac{1}{\alpha}+\frac{2}{\e}\right)I_{2d}\leq 
        &\begin{pmatrix}
            X^*&0\\0& Y^*
        \end{pmatrix}
        \leq 
        \left(\frac{1}{\e}+\frac{2\alpha}{\e^2} \right)
        \begin{pmatrix}
            I_{d}&-I_{d}\\-I_{d}&I_{d}
        \end{pmatrix}.
        \end{align}

        \medskip
        \noindent{\it Step 3: Obtaining infinite-dimensional jets.} For large enough $k \in \mathbb N$, $u_2(t_k^*,m_k^*) \geq u_2(t^*,m^*)-1$, denote by $A$ the set of $\theta = (t,\mu,m)$ such that 
        \begin{align*}
            u_1(\theta)=\tilde u(\theta)-\frac{1}{2\e}\Phi(\mu) \geq u_2(t^*,m^*)-1.
        \end{align*}
        According to the definition of $\tilde u$, we have 
        \begin{align*}
            \delta \vartheta(\theta)+\frac{1}{\e} \rho_F^2(\mu,\mu^*) \leq 1- u_2(t^*,m^*)-\frac{1}{2\e} \mathcal{L}(\mu,\mu^*,\nu^*)+ \lVert u \rVert_{\infty}. 
        \end{align*}
        Then applying Cauchy-Schwarz inequality to $\mathcal{L}(\mu,\mu^*,\nu^*)$, we get that $|\mathcal{L}(\mu,\mu^*,\nu^*)| \leq \frac{C}{2\epsilon}\rho_F(\mu^*,\nu^*)$ with some positive constant $C>0$. Therefore for any $\theta \in A$,  
        \begin{align*}
        \delta \vartheta(\theta)+\frac{1}{\e} \rho_F^2(\mu,\mu^*) \leq 1- u_2(t^*, m^*)+ \frac{C}{2\e}\rho_F(\mu^*,\nu^*)+\lVert u \rVert_{\infty},
        \end{align*}
        and hence the closure of $A$ is compact.
        
        We prove that for every large enough $k$ there exists $\mu_k^*$ such that with  $\theta_k^*:=(t_k^*,\mu_k^*,m_k^*)$       
        $$
            u_2(t_k^*,m_k^*)=u_1(\theta_k^*).
        $$ 
        Indeed, fixing any  $(t_k^*,m^*_k)$, take a sequence $(t_k^n,m_k^n) \to (t_k^*, m^*_k)$ as $n \to \infty$ such that 
        \begin{align*}
            u_2(t_k^*,m_k^*)=\lim\limits_{n\to \infty} \sup_{\mu \in \mathcal{P}_2(\R^d)} u_1(t_k^n,\mu,m_k^n). 
        \end{align*}
        Due to the compactness of $A$ and the upper semi-continuity of $u_1$, for each $n \in \mathbb N$ there exists $\mu_k^n \in \mathcal{P}_2(\mathbb R^d)$ that 
        \begin{align*}
            \sup_{\mu \in \mathcal{P}_2(\R^d)} 
            u_1(t_k^n,\mu, m_k^n)
                =
            u_1(t_k^n,\mu^n_k, m_k^n).
        \end{align*}
        Again by the compactness of $A$, there exists a limiting point $\mu_k^*$ of $(\mu_k^n)_{n \in \mathbb N}$, and hence 
        \begin{align*}
             u_2(t_k^*, m_k^*) = \lim\limits_{n \to \infty}u_1(t_k^n, \mu^n_k, m_k^n) \leq u_1(t_k^*, \mu_k^*, m_k^*), 
        \end{align*}
        which is actually an equality due to the definition of $u_2$. For the same token, there exists $\nu_k^*$ such that with $\iota_k^*:=(s_k^*,\nu_k^*,n_k^*)$,
        $$v_2(s_k^*,n_k^*)=v_1(\iota_k^*).$$
        
        Then because $(\theta^*,\iota^*)$ is the strict global maximum of \eqref{eq:optu1v1} and the upper-semicontinuity of \eqref{eq:optu1v1}, it must hold that $(\theta_k^*,\iota_k^*) \to (\theta^*,\iota^*)$ as $k \to \infty$.  According to the inequality \eqref{maxu2v2} and the definitions \eqref{eq:defu2} and \eqref{eq:defv2}, 
        \begin{align*}
        &u_1(\theta)-v_1(\iota) -h^*_k(t-t^*_k)-j^*_k(s-s^*_k)
            -\frac{1}{2}(m-m^*_k)^\top X^*_k(m-m^*_k)
                -\frac{1}{2}(n-n^*_k)^\top Y^*_k(n-n^*_k)\notag
        \end{align*}
        obtains a local maximum at $(\theta_k^*,\iota_k^*)$. Denoting 
        \begin{align*}
            \Phi_k(\theta)=&
                \frac{1}{2\e}\Phi(\mu)+h^*_k(t-t^*_k)
                + \frac{1}{2}(m-m^*_k)^\top X^*_k(m-m^*_k), \\
            \Psi_k(\iota)=&
                \frac{1}{2\e}\Psi(\nu)+ j^*_k(s-s^*_k)
                +\frac{1}{2}(n-n^*_k)^\top Y^*_k(n-n^*_k), 
        \end{align*}
        due to the definition of $u_1,v_1$, $\tilde u(\theta)-\tilde v(\iota)-  \Phi_k(\theta)- \Psi_k(\iota)$ obtains a local maximum at $(\theta_k^*,\iota_k^*)$. Therefore, according to \eqref{eq:tstarn}, the limit of derivatives of $\Phi_k(\theta_k^*), \Psi_k(\iota_k^*)$ is in the closure of jets of $\tilde u, \tilde v$ respectively, namely, 
        \begin{align*}
        &  \lim\limits_{k \to \infty}(\pa_t,  D_{\mu}, D_{x\mu}, \pa_{mm}^2) \Phi_k( \theta_k^*)\\
        & = \left(\frac{1}{\e}(t^*-s^*), \frac{D_\mu \Phi(\mu^*)(\cdot)}{2\e},\frac{D_{x\mu} \Phi(\mu^*)(\cdot)}{2\e},X^* \right) \in \bar J^{2,+} \tilde u(\theta^*), \\
        & \lim\limits_{k \to \infty}(\pa_t,  D_{\mu}, D_{x\mu}, \pa_{mm}^2) \Psi_k( \iota_k^*) \\
        & = \left(\frac{1}{\e}(t^*-s^*), -\frac{D_\mu \Psi(\nu^*)(\cdot)}{2\e},-\frac{D_{x\mu} \Psi(\nu^*)(\cdot)}{2\e},-Y^* \right) \in \bar J^{2,-} \tilde v (\iota^*). 
        \end{align*}
   \end{proof}

    Here is the assumption on $G$ for the comparison principle:
        \begin{assumption}\label{ass:comparison}
     \begin{itemize}
       \item[(i)] Assume that the extended Hamiltonian $G^e$ satisfies that for any $\theta = (t,\mu,m)$, $p_1,p_2 \in B_l^{d}$ and $q_1,q_2 \in B_l^{d \times d}$
        \begin{align*}
            |G^e(\theta,p_1,q_1,X_1)
            - G^e(\theta,p_2,q_2,X_2)|
            \le 
            L_G
            &\bigg(1 + |m| +
            \int_{\R^d}|x|\mu(dx)\bigg)
            \\&
            \Big(|p_1 - p_2|_l + |q_1 - q_2|_l  + |X_1 - X_2| \Big),
        \end{align*}
        with some positive constant $L_G$.
      
        \item[(ii)]
        There exists a modulus of continuity $\omega_G$ such that for all
        $\e > 0$, 
        $(\theta = (t,\mu, m),\iota = (s,\nu,n)) \in \big([0,T) \x \Pc_2(\R^d) \x \R^d\big)^2$, we have the inequality
        \begin{align*}
            &~
            G^e\Big(\theta, 
                \nabla\kappa,
                \nabla^2\kappa,
                X
            \Big)
            -
            G^e\Big(\iota, 
                \nabla\kappa,
                \nabla^2\kappa,
                -Y
            \Big) 
            \\ \le & ~
            \omega_G\bigg(\frac{1}{\e}d^2_F(\theta,\iota) + d_F(\theta,\iota)\bigg)  
            \bigg(1 + |m| + |n| + \int_{\R^d}|x|(\mu + \nu)(dx)\bigg),
        \end{align*}
        where the function $\kappa$ on $\R^d$ is defined as
        \begin{equation*}
            \kappa(x)
            ~ := ~
            \frac{1}{\e}\int_{\R^d}
            \frac{Re(F_k(\mu - \nu)f^*_k(x))}
            {(1 + |k|^2)^\lambda}
            dk,
           \quad
            x \in \R^d,
        \end{equation*}
        whenever $X \leq -Y$, $X,Y \in \S_d $.
        \end{itemize}
    \end{assumption}

    \begin{remark}
   We relax Assumption 3.1 in \cite{BaEkZh23} by replacing the term $\omega_G(d_F(\theta,\iota))$ in \cite[Assumption 3.1(ii)]{BaEkZh23} with $\omega_G\left(\frac{1}{\varepsilon} d^2_F(\theta,\iota) + d_F(\theta,\iota)\right)$. 
   
    \end{remark}
    \begin{theorem}[Comparison Principle]\label{thm:comparison}
            Assume that $u, -v: [0,T] \times \Pc_2({\R^d})  \rightarrow \R$ are bounded upper-semicontinuous functions, and $u$ (resp. $v$) is a viscosity subsolution (resp. supersolution) of the equation
    \begin{align}
        -\pa_t u(t,\mu)
        =
        G(t,\mu, D_\mu u (t,\mu),D_{x\mu}u(t\,\mu),\Hc u(t,\mu)).
    \end{align}
        Then, under Assumption \ref{ass:comparison}, $u(T,\cdot) \le v(T,\cdot)$ implies that $u \le v$ for all $(t,\mu) \in [0,T] \x \Pc_2({\R^d})$.
    \end{theorem}

    \begin{proof}


\medskip
 \noindent  \emph{Step 1: Preliminary.} 
 Recall that  $\ub(t,\mu,m) := u(t,(I_d+m)_\sharp\mu)$, $\vb(t,\mu,m) := v(t,(I_d+m)_\sharp\mu)$. We prove the result by contradiction. Assume that 
        \begin{equation*}
            \sup_{(t,\mu,m) \in [0,T] \x \Pc_2({\R^d}) \x \R^d}(\ub - \vb)(t,\mu,m) \geq 4r,
        \end{equation*}
        where $r$ is a strictly positive constant.
        Then we choose a $h > 0$ small enough depending on $r$ such that 
        $\bar u_h:= \bar u - h(T - t + 1)$ satisfies 
        \begin{equation*}
            \sup_{(t,\mu,m) \in [0,T] \x \Pc_2({\R^d}) \x \R^d}(\bar u_h - \bar v)(t,\mu,m) \ge 3r > 0.
        \end{equation*}
        We then verify that $\bar u_h$ is still a viscosity subsolution of the equation \eqref{eq:NewHJB}. Suppose $\phi$ is a partial $C^2$-regular test function such that $\ub_h - \phi$ has a local maximum at $(t^*,\mu^*,m^*)$. Equivalently, $\ub - (\phi + h(T - t + 1))$ has a local maximum at $(t^*,\mu^*,m^*)$. 
        We  write $\phi + h(T - t + 1)$ as $\phi_h$ for short.
        Then, the viscosity property of $u$ gives that, for any $(t^*,\mu^*,m^*)$,
        \begin{align*}
            &
            -\left(\pa_t \phi_h
                +
                G^e(\cdot, D_\mu \phi_h,D_{x\mu} \phi_h, \pa^2_{mm} \phi_h)\right)
            (t^*, \mu^*, m^*)
            \\  ~ = &~
            -\left(\pa_t \phi_h
                +
                G^e(\cdot, D_\mu \phi,D_{x\mu} \phi, \pa^2_{mm} \phi)\right)(t^*, \mu^*, m^*)
            \\ ~ = & ~
             -\left(\pa_t \phi
                +
                G^e(\cdot, D_\mu \phi,D_{x\mu} \phi, \pa^2_{mm} \phi)\right)(t^*, \mu^*, m^*)
                - h
            ~ \le ~ - h 
            ~ <   ~ 0.
        \end{align*} 

        \medskip
                
       \noindent \emph{Step 2: Doubling the variables.}
        Recall in \eqref{eq:vartheta} that for each $\theta = (t, \mu, m) \in [0,T] \times \Pc_2({\R^d}) \x \R^d $, recall that the function
        $\vartheta(\theta)=1+|m|^2+ \int_{{\R^d}}|x|^2\mu(dx)$, and  $\vartheta$ is lower-semicontinuous. 
        
        Now for any $\e $, $\delta$ > 0, and $(\theta, \iota) \in \left( [0,T] \x \Pc_2({\R^d}) \x \R^d  \right)^2 $, we define
        \begin{equation*}
            H_\e^\delta(\theta, \iota) 
            ~ := ~
            \ub_h(\theta) - \vb(\iota) - \frac{1}{2\e}d^2_F(\theta,\iota) - \delta\big(\vartheta(\theta) + \vartheta(\iota)\big),
        \end{equation*}
        and claim that the supremum of $H^\delta_\e$ over all admissible $(\theta,\iota)$ can be obtained.
        
        In fact, there exists an element $\theta_0 \in [0,T]  \x \Pc_2({\R^d}) \x \R^d$, such that
        \begin{equation*}
            H_\e^\delta(\theta_0,\theta_0) + 2\delta\vartheta(\theta_0)
            ~ = ~
            \ub_h(\theta_0) - \vb(\theta_0)
            ~ \ge ~ 2r > 0,
        \end{equation*}
        and for $\delta < \frac{r}{2\vartheta(\theta_0) + 1}$. So it holds that
        \begin{align*}
            \sup_{(\theta, \iota) \in ([0,T]   \x  \Pc_2({\R^d})\x \R^d)^2}
            H_\e^\delta(\theta, \iota)
            ~ \ge ~ 
            H_\e^\delta(\theta_0,\theta_0)
            ~ \ge ~
            r > 0.
        \end{align*}
        Take a sequence $\{(\theta_k,\iota_k)\}_{k \in \N_+}$ with $H_\e^\delta(\theta_k,\iota_k) > 0$ such that
        \begin{equation*}
            \lim_{k \to \infty}H_\e^\delta(\theta_k,\iota_k) 
            ~ = ~
            \sup_{(\theta, \iota) \in ([0,T]   \x  \Pc_2({\R^d})\x \R^d)^2}
            H_\e^\delta(\theta, \iota).
        \end{equation*}
        We have the estimation that for each $k \in \N_+$
        \begin{align*}
            \delta\big(\vartheta(\theta_k) + \vartheta(\iota_k)\big)
            &~ \le ~
            \ub_h(\theta_k) - \vb(\iota_k) - \frac{1}{2\e}d^2_F(\theta_k,\iota_k) \\
            &~ \le ~
            \ub_h(\theta_k) - \vb(\iota_k)
            ~ \le ~
            2M,
        \end{align*}
        where $M$ is a common upper bound for $u$ and $-v$.
        Since, for any $c > 0$, the sublevel set $\{\theta \in [0,T]  \x \Pc_2({\R^d})  \x \R^d: \vartheta(\theta) \le c\}$ is compact,  
        without loss of generality we can assume that $\{(\theta_k,\iota_k)\}_{k \in \N_+}$ is
        a convergent sequence in $[0,T] \x \Pc_2({\R^d}) \x \R^d$, and 
        $$
            \lim_{k \to \infty}
            (\theta_k, \iota_k)
            ~ = ~
            (\theta^\delta_\e,\iota^\delta_\e),
        $$
        for some $(\theta^\delta_\e,\iota^\delta_\e) \in ([0,T] \x \Pc_2({\R^d}) \x \R^d)^2$.
        As a consequence, the upper-semicontinuity of $H^\delta_\e $ yields that
        \begin{align*}
            &
            H_\e^\delta(\theta^\delta_\e,\iota^\delta_\e)
            ~ = ~
            \sup_{(\theta, \iota) \in ([0,T]  \x  \Pc_2({\R^d}) \x \R^d)^2}
            H_\e^\delta(\theta, \iota)
            ~>~ 0,
            \\~
            &
            \delta\big(\vartheta(\theta^\delta_\e) 
             + \vartheta\big(\iota^\delta_\e)\big)
             ~ \le ~
             \varliminf_{k \to \infty}
                \delta\big(\vartheta(\theta_k) 
                + \vartheta\big(\iota_k)\big)
             ~ \le ~
             2M.
        \end{align*}
        Again, without loss of generality we can assume that $\{(\theta^\delta_\e,\iota^\delta_\e)\}_{\e \in \N_+}$ is
        a convergent sequence in $[0,T] \x \Pc_2({\R^d}) \x \R^d$, and 
        $$
            \lim_{\e \to 0+}
            (\theta^\delta_\e,\iota^\delta_\e)
            ~ = ~
            (\theta^\delta,\iota^\delta),
        $$
        for some $(\theta^\delta,\iota^\delta) \in ([0,T] \x \Pc_2({\R^d}) \x \R^d)^2$.

        Then, we claim that 
        \begin{equation}\label{eq:penalGoZero}
            \lim_{\e \to 0}
            \frac{1}{2\e}d^2_F(\theta^\delta_\e,\iota^\delta_\e)
            ~ = ~ 0.
        \end{equation}
        In fact, one may observe that
        \begin{align*}
             \varlimsup_{\e \to 0+}
             \frac{1}{2\e}d^2_F(\theta^\delta_\e,\iota^\delta_\e)
             & ~ =  ~
             \varlimsup_{\e \to 0+}
             \Big(- H_\e^\delta(\theta^\delta_\e,\iota^\delta_\e)
             - \delta\big(\vartheta(\theta^\delta_\e) 
                + \vartheta\big(\iota^\delta_\e)\big)
             + \bar u_h(\theta^\delta_\e) - \bar v(\iota^\delta_\e)\Big)
             \\ & ~ \le ~ 
              \varlimsup_{\e \to 0+}
              \Big(\ub_h(\theta^\delta_\e) - \vb(\iota^\delta_\e)\Big)
             ~ \le ~
             2M,
        \end{align*}
        and obtain that
        \begin{equation*}
            \varlimsup_{\e \to 0+} 
            d^2_F(\theta^\delta_\e,\iota^\delta_\e)
            ~ = ~
            0,
            ~\mbox{i.e.}~
            d^2_F(\theta^\delta,\iota^\delta)
            ~ = ~ 0,
            ~\mbox{or}~
            \theta^\delta 
            ~ = ~
            \iota^\delta.
        \end{equation*}
        Then it follows that
        \begin{align*}
            \varlimsup_{\e \to 0+}
             \frac{1}{2\e}d^2_F(\theta^\delta_\e,\iota^\delta_\e)
             & ~ =  ~
             \varlimsup_{\e \to 0+}
             \Big(- H_\e^\delta(\theta^\delta_\e,\iota^\delta_\e)
             - \delta\big(\vartheta(\theta^\delta_\e) 
                + \vartheta\big(\iota^\delta_\e)\big)
             + \ub_h(\theta^\delta_\e) - \vb(\iota^\delta_\e)\Big)
             \\ & ~ \le ~ 
             \varlimsup_{\e \to 0+}
             \Big(- \sup_{\theta \in [0,T] \x \Pc_2(\R^d) \x \R^d}H_\e^\delta(\theta,\theta)
             - \delta\big(\vartheta(\theta^\delta_\e) 
                + \vartheta\big(\iota^\delta_\e)\big)
             + \ub_h(\theta^\delta_\e) - \vb(\iota^\delta_\e)\Big)
             \\ & ~ \le ~ 
             - \sup_{\theta \in [0,T] \x \Pc_2(\R^d) \x \R^d}H_\e^\delta(\theta,\theta)
             - \delta\big(\vartheta(\theta^\delta) 
                + \vartheta\big(\iota^\delta)\big)
             + \ub_h(\theta^\delta) - \vb(\iota^\delta)
             \\ & ~ \le ~ 
             - \sup_{\theta \in [0,T] \x \Pc_2(\R^d) \x \R^d}H_\e^\delta(\theta,\theta)
             +  H_\e^\delta(\theta^\delta,\theta^\delta)
             ~ \le ~ 0,
        \end{align*}
        where the second inequality follows from the fact that $\sup_{\theta \in [0,T] \x \Pc_2(\R^d) \x \R^d}H_\e^\delta(\theta,\theta)$ is independent of $\e $, and
        that $\ub_h$,$-\vb$, and $-\vartheta$ are upper semi-continuous. 
        
        Moreover, it is clear that $t^\delta \neq T$, where $\theta^\delta = (t^\delta,\mu^\delta, m^\delta)$.
        Otherwise, it follows that
        \begin{equation*}
            0 < r \le
            \varlimsup_{\e \to 0+}
            H^\delta_\e (\theta^\delta_\e,\iota^\delta_\e)
            ~ \le ~
            \varlimsup_{\e \to 0+}
            \ub_h(\theta^\delta_\e) - \vb(\iota^\delta_\e)
            ~ \le ~
            \ub_h(\theta^\delta) - \vb(\iota^\delta)
            ~ \le ~ 
            0.
        \end{equation*}
\medskip
        \noindent \emph{Step 3: Contradiction by viscosity property and estimation.} Define 
        $
             \ut_h := \ub_h - \delta\vartheta,
             ~
             \vt := \vb + \delta\vartheta.
        $ 
        Without loss of generality, we assume that 
        $(\theta^\delta_{\e},\iota^\delta_{\e})$
        is a strict global maximum of 
        $$
            (\theta,\iota) \longmapsto H^\delta_\e (\theta,\iota) 
        = \ut_h(\theta) - \vt(\iota) - \frac{1}{2\e}d^2_F(\theta,\iota).
        $$
        Then by Ishii's lemma \ref{lemm:ishii}, it follows that 
        for any $\e>0$, there exist $X^*,Y^* \in \Sc_{d}$ such that 
        \begin{align*}
         & \left(\frac{1}{\e}(t^*-s^*), 
            \frac{D_\mu \Phi(\mu^*)(\cdot)}{2\e},
            \frac{D_{x\mu} \Phi(\mu^*)(\cdot)}{2\e},X^* \right) 
            \in \bar J^{2,+} \tilde u_h(\theta^*), \\
        & \left(\frac{1}{\e}(t^*-s^*), 
            -\frac{D_\mu \Psi(\nu^*)(\cdot)}{2\e},
            -\frac{D_{x\mu} \Psi(\nu^*)(\cdot)}{2\e},-Y^* \right) 
            \in \bar J^{2,-} \tilde v (\iota^*),
            \end{align*}
            as well as 
        \begin{align*}
            -\left(\frac{1}{\alpha}+\frac{2}{\e}\right)I_{2d}\leq 
            &\begin{pmatrix}
                X^*&0\\0& Y^*
            \end{pmatrix}
            \leq 
            \left(\frac{1}{\e}+\frac{2\alpha}{\e^2} \right)
            \begin{pmatrix}
                I_{d}&-I_{d}\\-I_{d}&I_{d}
            \end{pmatrix}.
        \end{align*}
        For a better statement, let us recall and define some notations: 
        \begin{align*}
            &
            \theta^\delta_\e 
            ~ = ~
            (t^\delta_\e,\mu^\delta_\e,m^\delta_\e),
            ~
            \iota^\delta_\e 
            ~ = ~
            (s^\delta_\e,\nu^\delta_\e,n^\delta_\e), \\ &
            q(x) := |x|^2, \quad \quad
            x \in \R^d, 
            \\ &
            \kappa^\delta_\e (x)
            ~ := ~
            \frac{1}{\e}\int_{\R^d}
            \frac{Re(F_k(\mu^\delta_\e - \nu^\delta_\e)f^*_k(x))}
            {(1 + |k|^2)^\lambda}
            dk,
             ~
            x \in \R^d.
        \end{align*}
       Direct calculation yields that for $\theta = (t,\mu, m) \in [0,T] \x \Pc_2(\R^d)  \x \R^d$,
        \begin{align*}
            &
            D_\mu\vartheta(\theta,x)
            ~ = ~
             \nabla q(x),
            \quad
            D_{x\mu}\vartheta(\theta,x)
            ~ = ~
            \nabla^2 q(x),
            \quad
            \nabla_m^2\vartheta(\theta)
            ~ = ~
            \nabla^2 q(m),
            \\
            &
            \frac{D_\mu \Phi(\mu^*)(\cdot)}{2\e}
            ~ = ~
            -\frac{D_\mu \Psi(\nu^*)(\cdot)}{2\e}
            ~ = ~
            \nabla\kappa^\delta_\e(\cdot),
            \quad
            \frac{D_{x\mu} \Phi(\mu^*)(\cdot)}{2\e}
            ~ = ~
            -\frac{D_{x\mu} \Psi(\nu^*)(\cdot)}{2\e}
            ~ = ~
            \nabla^2\kappa^\delta_\e(\cdot).
        \end{align*}
        For convenience, we denote by
        \begin{align*}
            &\alpha^\delta_\e 
            ~:= ~
            \bigg(
                \nabla\kappa^\delta_\e + \delta \nabla q,
                \nabla^2\kappa^\delta_\e + \delta \nabla^2 q,
               X^* + \delta \nabla^2 q(m^\delta_\e)
            \bigg),
            \\~&
            \beta^\delta_\e 
            ~:= ~
            \bigg(
                \nabla\kappa^\delta_\e - \delta \nabla q,
                \nabla^2\kappa^\delta_\e - \delta \nabla^2 q,
               -Y^* - \delta \nabla^2 q(n^\delta_\e)
            \bigg),
        \end{align*}
        Then,  the mapping
        $ \theta \longmapsto (D_{\mu}\vartheta, D_{x \mu} \vartheta, \nabla^2_m \vartheta)(\theta) $ is continuous, and it follows from the linearity and the definition of the closure of the jets that
        \begin{align*}
            &
            \bigg(
                \frac{1}{\e}(t^\delta_\e - s^\delta_\e),
                \alpha^\delta_\e 
            \bigg)
            \in
            \Jb^{2,+}\ub_h(\theta^\delta_\e),
            ~
            \bigg(
                \frac{1}{\e}(t^\delta_\e - s^\delta_\e),
                \beta^\delta_\e 
            \bigg)
            \in
            \Jb^{2,-}\vb(\theta^\delta_\e),
        \end{align*}
        By viscosity property of $\bar u_h$, and $\bar v$, one has that
        \begin{align*}
            -
            \frac{1}{\e}(t^\delta_\e - s^\delta_\e)
            -  
            G^e\Big(\theta^\delta_\e, \alpha^\delta_\e 
            \Big) 
            ~ \le 
            -h,
            \quad
            -
            \frac{1}{\e}(t^\delta_\e - s^\delta_\e)
            -  
            G^e\Big(\iota^\delta_\e, \beta^\delta_\e 
            \Big) 
            \ge 
            0,
        \end{align*}
        i.e.,
        \begin{align*}
            h \le G^e\Big(\theta^\delta_\e, 
                \alpha^\delta_\e 
            \Big)
            -
            G^e\Big(\iota^\delta_\e, 
                \beta^\delta_\e 
            \Big),
        \end{align*}
        and on the other hand, one can obtain the following estimation
        \begin{align*}
            &
            G^e\Big(\theta^\delta_\e, 
                \alpha^\delta_\e 
            \Big)
            -
            G^e\Big(\iota^\delta_\e, 
                \beta^\delta_\e 
            \Big) 
            \\ \le ~ &
            G^e\Big(\theta^\delta_\e, 
                \nabla\kappa^\delta_\e,
                \nabla^2\kappa^\delta_\e,
                X^*
            \Big)
            -
            G^e\Big(\iota^\delta_\e, 
                \nabla\kappa^\delta_\e,
                \nabla^2\kappa^\delta_\e,
                -Y^*
            \Big) 
                \\ & +
            L_G\bigg(2 + |m^\delta_\e| + |n^\delta_\e| + \int_{\R^d}|x|(\mu^\delta_\e + \nu^\delta_\e)(dx)
            \bigg)
            (4d + 2\sqrt{d})\delta 
            \\ \le ~ &
            \Bigg(
            \omega_G\bigg(\frac{1}{\e}d^2_F(\theta^\delta_\e,\iota^\delta_\e) + d_F(\theta^\delta_\e,\iota^\delta_\e)\bigg)
            +
            L_G(4d + 2\sqrt{d})\delta
            \Bigg)
            \bigg(2 + |m^\delta_\e | + |n^\delta_\e | + \int_{\R^d}|x|(\mu^\delta_\e + \nu^\delta_\e)(dx)
            \bigg),
        \end{align*}
        where the first inequality holds true by Assumption \ref{ass:comparison} (i), and the second one by Assumption \ref{ass:comparison} (ii).
        Letting $\e $ go to $0$, one obtains
        \begin{align*}
             h 
            ~ \le ~
            4(2d + \sqrt{d})L_G
            \bigg(1 + |m^\delta| + \int_{\R^d}|x|\mu^\delta(dx)
            \bigg) \delta .
        \end{align*}
        Finally, letting $\delta$ go to $0$, one gets the desired contradiction. 
    \end{proof}
\section{Applications}\label{sec:appli}
\subsection{An application to controlled stochastic filtering problems}\label{sec:appli-1}
As an application, we show that the value function of a stochastic control problem with partial observation is the unique viscosity solution of a parabolic equation. The setting is almost the same as in \cite{BaEkZh23}, but we prove the uniqueness under much weaker assumptions. 

Suppose that we have two independent Brownian motions $B,W$ of dimensions $d_1$ and $d_2$ on a  filtered probability space $(\Omega, \mathcal{F},\mathbb P)$, and a compact closed control set $A \subset \R^d$. Take the coefficients $b:\R^d \x \Pc(\R^d)\times A \to \R^d$, $\sigma:\R^d\x \Pc(\R^d) \times A \to \R^{d\times d_1}$, and $\tilde\sigma:A \to \R^{d \times d_2}$. 
The space of probability measure $\Pc(\R^d)$ is equipped with the Fourier-Wasserstein distance $\rho_F$.
We consider the following McKean-Vlasov stochastic differential equations
\begin{align*}
dX^{t,\mu,\alpha}_s&=b(X^{t,\mu,\alpha}_s,m^{t,\mu,\alpha}_s,\alpha_s) \, ds + \sigma(X^{t,\mu,\alpha}_s,m^{t,\mu,\alpha}_s,\alpha_s) \, dB_s + \tilde \sigma(\alpha_s) \, dW_s, \quad \text{$t \leq s \leq T$}, \\
X_t^{t,\mu,\alpha}&=\xi, 
~
m^{t,\mu,\alpha}_s := \Lc(X^{t,\mu,\alpha}_s \, |\,\Fc^W_s),
\end{align*}
where $\xi$ is independent of $B,W$ with distribution $\mu \in \Pc_2(\R^d)$  and $\alpha: \Omega \x [0,T] \rightarrow A$ is an admissible control adapted only to the filtration generated by $W$. Since $\xi$ is independent of $B,W$, it can be easily checked that the distribution of $(X_s^{t,\mu,\alpha},\alpha_s)$ is independent of the choice of $\xi$.

Then $m_t$, the conditional law of state $X$ at time $t$, satisfies the equation
\begin{align}\label{eq:conditionlaw}
    d \langle m^{t,\mu,\alpha}_s, f \rangle &=  \langle m^{t,\mu,\a}_s, L^{\alpha_s} f \rangle \,ds + \langle m^{t,\mu,\a}_s, M^{\a_s} f \rangle \, dW_s, \quad \text{ $t \leq s \leq T$}  \\
    m^{t,\mu,\alpha}_t &= \mu, \notag
\end{align}
where $f:\R^d \x \Pc(\R^d) \longrightarrow \R$ is any $C^2$ test function and 
\begin{align*}
    L^a f(\cdot):=&  b(\cdot,a)^\top \pa_xf(\cdot) +\frac{1}{2} Tr \left(  (\sigma\sigma^{\top}(\cdot,a)  +\tilde\sigma \tilde \sigma^{\top}(a)) \pa_x^2  f(\cdot)  \right),\\
    M^a f(\cdot):=&   \tilde\sigma(a)^\top \pa_xf(\cdot). 
\end{align*}

Take $\mathcal{A}:=\{\alpha=(\alpha_s)_{0\leq s \leq T}: \alpha : \Om \x [0,T] \longrightarrow A \text{ is measurable and $\F^W$ adapted} \}$ to be the set of all admissible controls. Given a running cost $r:\R^d \times A \to \R$, and a terminal cost $l: \R^d \to \R$, we define the cost of control $\alpha \in \mathcal{A}$,
\begin{align*}
    J(t,\mu, \alpha):=\E \left[\int_t^T r(X^{t,\mu,\alpha}_s, m^{t,\mu,\alpha}_s, \alpha_s) \, ds + l(X^{t,\mu,\alpha}_T)  \right].
\end{align*}
We aim at solving the following optimization problem 
\begin{align*}
    v(t,\mu)=\inf_{\alpha \in \mathcal{A}} J(t,\mu,\alpha). 
\end{align*}

\begin{assumption}\label{assumption1}
\begin{itemize}
    \item[(i)] The functions $b,\sigma,\tilde\sigma, r,l$ are bounded and Lipschitz continuous in their domains. 
    \item[(ii)] $\sigma$ is uniform elliptic, i.e. there exists some positive constant $\delta$ such that $ |\sigma(x,\mu,a) \xi |^2  \geq \delta |\xi|^2$ for any $x,a,\xi,\mu$. 
\item[(iii)] For any $a \in A$, $\mu \in \Pc(\R^d)$, $ x\mapsto \sigma^\top(x,\mu,a)\sigma(x,\mu,a) \in H^{\lambda+d/2+1}(\R^d)$,
    $ x\mapsto b(x,\mu,a) \in H^{\lambda}(\R^d)$,
    $ x\mapsto r(x,\mu,a) \in H^{\lambda}(\R^d)$,
where $H^s(\R^d)$ with $s > 0$ is a fractional Sobolev space (see Definition \ref{def:Sobolev} for details), and  $\lambda$ is given by   
    \begin{equation} \label{eq:defn-lambda}
        \lambda = 
        \left\{
        \begin{split}
        &\lfloor\frac{d}{2}\rfloor + 4,
        ~ \mbox{for}~ d = 4k, 4k + 1,
        \\
        &\lfloor\frac{d}{2}\rfloor + 3,
        ~ \mbox{for}~ d = 4k + 2, 4k + 3.
        \end{split}
    \right.
    \end{equation}
\end{itemize}
\end{assumption}

Let us define for $(a,\mu,p,q,M) \in A  \times \Pc_2(\R^d) \times B^d_{l} \times B^{d \times d}_{l} \times \mathbb{S}^d $
\begin{align*}
    K(a,\mu,p,q,M):=& \int  r(x,\mu,a)+ b(x,\mu,a)^\top p(x) +\frac{1}{2} Tr\left(q(x) \sigma \sigma^\top (x,\mu,a) \right)\, \mu(dx) \\
    &+\frac{1}{2} Tr( \tilde \sigma \tilde \sigma^\top(a) M).
\end{align*}
\begin{remark}\label{rmk:filter}
    $\mathrm{(i)}$
    Compared to Assumption 4.1 in \cite{BaEkZh23}, which requires exponential decay of $r$ and $l$'s derivatives, the above assumption is much weaker in $r$ and $l$, but requires uniform ellipticity in $\sigma$.
    
    \noindent $\mathrm{(ii)}$
    The above McKean-Vlasov dynamics under given assumption has a unique strong solution, we may refer to Theorem A.3 in \cite{10.1214/21-AOP1548}. 
\end{remark}

The proof of the following theorem is postponed to Subsection \ref{Proof:thm:viscosity}.
\begin{theorem}[Viscosity solution]\label{thm:viscosity_property}
    Under Assumption~\ref{assumption1}, the value function is continuous and is the unique viscosity solution of the equation 
    \begin{align}\label{eq:filtering}
        -\pa_t v(t,\mu)=& \inf_{a \in A} K(a, \mu, D_{\mu} v(t,\mu), D_{x \mu} v(t,\mu), \mathcal{H}v(t,\mu)) \notag \\
        v(T,\mu)=&\mu(l). 
    \end{align}
\end{theorem}

To prove uniqueness, we verify that the Hamiltonian of \eqref{eq:filtering} satisfies Assumption~\ref{ass:comparison}. For that purpose, we need the  commutator estimates inspired by \cite{gozzi2000hamilton} in Subsection 5.1, which is the key technical result of this paper. We would like to mention that these types of estimates have been used in \cite{bayraktar2024convergence} and an earlier ArXiv version of \cite{MeYa23} when $\sigma$ is constant. Following \cite[Corollary 5.6]{WDPP} line by line, one can show that the upper semicontinuous (lower semicontinuous) envelop $v^*$ ($v_*$) is a viscosity subsolution (supersolution) to \eqref{eq:filtering}. Then by the comparison principle, we have $v^*=v_*$, and hence $v$ is the unique continuous viscosity solution.

\subsection{An application to prediction problems under partial monitoring}

In \cite{JMLR:v24:22-1001}, the authors investigated an adversarial learning problem, which can be formulated as a zero-sum game between a forecaster and an adversary. At each round, the forecaster chooses an action between $K \ge 2$ alternative actions based on his/her partial observations, aiming at performing as well as the best constant strategy, while the adversary aims at maximizing the forecaster's regret. 

In particular, suppose that there are $K$ actions. At each round $t$, the forecaster chooses an action $I_t \in [K] := \{1, 2, \cdots, K\}$, and the adversary chooses a set of winning actions $J_t \subset [K]$.
The total gain of the forecaster $G_t$ and the total gain of $G^i_t$ of action $i$ evolve as
\begin{align*}
    G_{t + 1} - G_t &= \mathds{1}_{I_t \in J_t},
    \\
    G^i_{t+1} - G^i_t &= \mathds{1}_{i \in J_t},
    ~ i = 1, 2, \cdots, K.
\end{align*}
The goal of the forecaster is to design a robust strategy that performs as well as the best constant strategy under any adversarial environment, i.e., to minimize $\max_{Adversary} \E[\max_i (G^i_t - G_t)]$, where \textit{Adversary} denotes the set of all possible adversarial environments.
Both the forecaster and the adversary are allowed to adopt randomized strategies. At each round, they decide on distribution $b_t \in \Pc([K])$ of $I_t$ and $a_t \in \Pc(\{0,1\}^K)$ of $J_t$ respectively.

Let us now describe information observed by the forecaster and his/her admissible strategies in the partial information problem.
At initial time $t = 0$, both the adversary and the forecaster get informed of the distribution $m_0$ of $(G^1_0 - G_0,G^2_0 - G_0, \cdots, G^K_0 - G_0)$. For any $t \ge 0$, we use the random variable $Y_t$ to indicate whether the forecaster makes a good decision or not, where
\begin{align*}
    Y_t := \mathds{1}_{I_t \in J_t} I_t - \mathds{1}_{I_t \notin J_t}I_t.
\end{align*}
Both players can observe the law of adversary's control $a_{t- 1}$, i.e. the choice of $J_{t - 1}$, and the indicator $y_{t - 1}$ which is a realization of $Y_{t-1}$ . Their accumulated information is given by 
\begin{align*}
    h_t := (m_0,a_0,y_0, \cdots,a_{t-1}, y_{t-1}) \in \Hc_t := \Pc(\R^K) \x \Big(\Pc(\{0,1\}^K) \x \{\pm i, i \in [K]\} \Big)^t.
\end{align*}
Then, the strategies of the forecaster and the adversary are measurable functions $\beta_t : \Hc_t \longrightarrow \Pc([K])$ and $\alpha_t : \Hc_t \longrightarrow \Pc(\{0,1\}^K)$ respectively. Define $\Ac$ to be the set of all possible strategies $\alpha := (\alpha_0, \alpha_1, \cdots, \alpha_{T-1})$ and $\Bc$ similarly.

Suppose this game starts from time $t$ with an initial distribution $\mu \in \Pc(\R^K)$. 
Then given any strategies $\alpha \in \Ac, \beta \in \Bc$, the regret for the forecaster is given by
\begin{align*}
    \gamma_T(t,\mu,\alpha,\beta)
    :=
    \E^{\mu,\alpha,\beta}[\max_i (G^i_T - G_T)| G^i_t - G_t \sim \mu].
\end{align*}
The goal of the forecaster is to solve
\begin{align*}
    v_T(t,\mu) := \inf_{\beta \in \Bc} \sup_{\alpha \in \Ac} \gamma_T(t,\mu,\alpha,\beta).
\end{align*}

Let us define the rescaled value functions
\begin{align*}
    u^T(s,\mu) 
    ~:= ~
    \frac{1}{\sqrt{T}}v_T(\lceil sT \rceil, (\sqrt{T}I_K)_\sharp\mu),
\end{align*}
and assume that $u^T$ converges to a function $u$, with $(t,\mu,m) \longrightarrow u(t,(I_K + m)_\sharp\mu)$ partial $C^2$-regular. Then $u$ should satisfy the PDE 
\begin{equation}\label{eq:regretPDE_original}
    \begin{split}
        -& \pa_tu(t,m) - \sup_{a \in \Pc(\{0,1\}^K)} \sum_i
        \\ &
        \frac{1}{2}u_i(t,m)\hat{a}(i)\bigg( \Vc_{a,i}^TD^2_{\mu\mu}u(t,\mu,[\mu],[\mu])\Vc_{a,i}
            ~ + \sum_{j : i \in j}\frac{a(j)}{\hat{a}(i)}e^T_{j^C}D_{x\mu}u(t,\mu,[\mu])e_{j^C}\bigg)
        \\ 
        + & \frac{1}{2}u_i(t,m)\hat{a}(-i)\bigg( \Vc_{a,-i}^TD^2_{\mu\mu}u(t,\mu,[\mu],[\mu])\Vc_{a,-i}
            ~ + \sum_{j : i \notin j}\frac{a(j)}{\hat{a}(-i)}e^T_{j}D_{x\mu}u(t,\mu,[\mu])e_{j}\bigg)
        ~ = ~ 0,
    \end{split}
\end{equation}
where $\{e_i : i= 1,\cdots,K\}$ is the canonical basis of $\R^K$, for any $j \subset [K]$, 
$e_j := \sum_{i:i \in j}e_i$
\begin{align*}
    &
    u_i(t,m)=\frac{d}{d\e}\big|_{\e=0} u(t, (I_d+\e e_i)_{\#} m),
    ~\forall (i,t,m) \in [K] \x [0,1] \x \Pc_2(\R^K)
    \\ &
    \hat{a}(i) := \sum_{j: i\in j}a(j),
    ~
    \hat{a}(-i) := \sum_{j: i \notin j}a(j), 
    ~\forall i \in [K],~ a \in \Pc(\{0,1\}^K),
    \\ &
    \Vc_{a,i} := \sum_{j: i \in j}\frac{a(j)}{\hat{a}(i)}e_{j^C},
    ~
    \Vc_{a,-i} := \sum_{j: i \notin j}\frac{a(j)}{\hat{a}(-i)}e_{j},
    \\ &
    D_\mu u(\mu,[\mu]) := \int_{\R^K}D_\mu u(\mu,x)\,\mu(dx),
    ~\forall \mu,\in  \Pc_2(\R^K),
    \\ &
    D^2_{\mu\mu}u(\nu,[\mu],[\mu]) := \int_{\R^K}\int_{\R^K}D^2_{\mu\mu} u(\mu,x,y)\,\mu(dx)\mu(dy),
    ~\forall \mu \in  \Pc_2(\R^K).
\end{align*}
Note that in the above PDE \eqref{eq:regretPDE_original} involves the product of $u_i$ and $D_{\mu\mu} u$, the Hamiltonian may become discontinuous in $u_i$ when $D^2_{\mu\mu} u$ explodes as $t \to 1$.

Thus it is more convenient to use the equation
\begin{equation}\label{eq:regretPDE}
    \begin{split}
                -& \pa_tu(t,m) - \sup_{i,a \in \Pc(\{0,1\}^K)} 
        \\ &
        \frac{1}{2}\hat{a}(i)\bigg( \Vc_{a,i}^TD^2_{\mu\mu}u(t,\mu,[\mu],[\mu])\Vc_{a,i}
            ~ + \sum_{j : i \in j}\frac{a(j)}{\hat{a}(i)}e^T_{j^C}D_{x\mu}u(t,\mu,[\mu])e_{j^C}\bigg)
        \\ 
        + & \frac{1}{2}\hat{a}(-i)\bigg( \Vc_{a,-i}^TD^2_{\mu\mu}u(t,\mu,[\mu],[\mu])\Vc_{a,-i}
            ~ + \sum_{j : i \notin j}\frac{a(j)}{\hat{a}(-i)}e^T_{j}D_{x\mu}u(t,\mu,[\mu])e_{j}\bigg)
        ~ = ~ 0,
    \end{split}
\end{equation}
to obtain regret bounds.

\begin{theorem}
     Assume that $u, -v: [0,T] \times \Pc_2({\R^K})  \rightarrow \R$ are bounded upper-semicontinuous functions, and $u$ (resp. $v$) is a viscosity subsolution (resp. supersolution) of the equation \eqref{eq:regretPDE}.

     Then, $u(1,\cdot) \le v(1,\cdot)$ implies that $u \le v$ for all $(t,\mu) \in [0,1] \x \Pc_2(\R^K)$.
\end{theorem}
\begin{proof}
    It's sufficient to verify that the Hamiltonian
    \begin{align*}
        G : \Pc_2(\R^K) \x B_l^{K \x K} \x \S_K  &\longrightarrow \R
        \\ 
        (\mu,q,M) & \longmapsto \sup_{i,a \in \Pc(\{0,1\}^K)} K(i,a,\mu,q,M),
    \end{align*}
    where
    \begin{align*}
        K(i,a,\mu,q,M) := ~
        &
                \frac{1}{2}\hat{a}(i)\bigg( \Vc_{a,i}^TM \Vc_{a,i}
            ~ + \sum_{j : i \in j}\frac{a(j)}{\hat{a}(i)}e^T_{j^C}\int_{\R^K} q\, \mu(dx)(e_{j^C} - \Vc_{a,i})\bigg)
        \\ 
        + & \frac{1}{2}\hat{a}(-i)\bigg( \Vc_{a,-i}^TM\Vc_{a,-i}
            ~ + \sum_{j : i \notin j}\frac{a(j)}{\hat{a}(-i)}e^T_{j}\int_{\R^K} q\, \mu(dx)(e_{j} - \Vc_{a,-i})\bigg),      
    \end{align*}
    satisfies the Assumption \ref{ass:comparison}, 
    since for $(\mu,m,q,M) \in \Pc_2(\R^K) \x \R^K \x  B_l^{K \x K} \x \S_K$, 
    \begin{align*}
        G^e(\mu,m,q,M) = & ~ G((I_K + m)_\sharp\mu,q(\cdot -m),M)
        \\ = & \sup_{i,a \in \Pc(\{0,1\}^K)} K(i,a,(I_K + m)_\sharp\mu,q(\cdot -m),M)
        \\ = & \sup_{i,a \in \Pc(\{0,1\}^K)} K(i,a,\mu,q,M)
        ~ = ~ G(\mu,q,M).
    \end{align*}

    For Assumption \ref{ass:comparison} (i), we consider any $\mu \in \Pc_2(\R^K)$, and $q_1,q_2 \in B_l^{K\x K}$, $M_1, M_2 \in \S_K$,
    \begin{align*}
        & ~
        |G(\mu,q_1,M_1) - G(\mu,q_2,M_2)| 
        \\~ \le & ~
        \sup_{i,a \in \Pc(\{0,1\}^K)} |K(i,a,q_1,M_1) - K(i,a,q_2,M_2)|
        \\~ \le & ~
        \sup_{i,a \in \Pc(\{0,1\}^K)} \frac{1}{2}\Big(\hat{a}(i)|\Vc_{a,i}|^2 + \hat{a}(-i)|\Vc_{a,-i}|^2\Big)\|M_1 - M_2\|
        \\& +  ~ \frac{1}{2}\bigg( 1 + \int_{\R^K}|x| \, \mu(dx)\bigg)
        \Big(\hat{a}(i)\big(|\Vc_{a,i}|^2 + 2^{K-1}\big) + \hat{a}(-i)\big(|\Vc_{a,-i}|^2 + 2^{K-1}\big)\Big)|q_1 - q_2|_l^{d \x d}
        \\ ~ \le & ~
        2^{3K - 2}\bigg( 1 + \int_{\R^K}|x| \, \mu(dx)\bigg)(|q_1 - q_2|_l^{d \x d} + \|M_1 - M_2\|),
    \end{align*}
    where the second inequality holds true by the fact that the number of subsets of $\{0,1\}^K$ that contain the element $i$ is $2^{K - 1}$,  
    the third inequality by that $|\Vc_{a,i}|, |\Vc_{a,-i}|, |\hat{a}(i)|, |\hat{a}(-i)|  \le 2^{K - 1}$.

    For Assumption \ref{ass:comparison} (ii), recalling that  
    for $\mu, \nu \in \Pc(\R^K)$
    \begin{align*}
          \kappa(x)
                ~ := ~
                \frac{1}{\e}\int_{\R^K}
                \frac{Re(F_k(\mu - \nu)f^*_k(x))}
                {(1 + |k|^2)^\lambda}
                dk,
        \\
        \int_{\R^K} \nabla^2\kappa\, (\mu - \nu)(dx)
        ~ := ~
        - \frac{1}{\e}\int_{\R^K}
                \frac{|F_k(\mu - \nu)|^2kk^T}
                {(1 + |k|^2)^\lambda}
                dk,
    \end{align*}
    and $M \longmapsto K(i,a,\mu,q,M)$ is increasing, it is sufficient to prove that
    \begin{align*}
        &~ K(i,a,\mu,\nabla^2\kappa, M) - K(i,a,\nu,\nabla^2\kappa, M) 
        \\ = ~& ~
        \frac{1}{2}\bigg(\sum_{j : i \in j}a(j)e^T_{j^C}\int_{\R^K} \nabla^2\kappa(x)\, (\mu - \nu)(dx)(e_{j^C} - \Vc_{a,i})
        \\ & \quad +
        \sum_{j : i \notin j}a(j)e^T_{j}\int_{\R^K} \nabla^2\kappa(x)\, (\mu - \nu)(dx)(e_{j} - \Vc_{a,-i})
        \bigg)
        \\ = ~ & ~
        - \frac{1}{\e}\int_{\R^K}
            \frac{|F_k(\mu - \nu)|^2}{(1 + |k|^2)^\lambda}
            \bigg(\sum_{j: i \in j}a(j)|k^Te_{j^C}|^2 
                -\hat{a}(i)|k^T\Vc_{a,i}|^2 \bigg)
            dk
            \\ & ~ - \frac{1}{\e}\int_{\R^K}
            \frac{|F_k(\mu - \nu)|^2}{(1 + |k|^2)^\lambda}
            \bigg(\sum_{j: i \notin j}a(j)|k^Te_{j}|^2 
                -\hat{a}(-i)|k^T\Vc_{a,-i}|^2 \bigg) 
            dk
        ~ \le ~ 0,
    \end{align*}
    where the inequality holds true by Cauchy-Schwarz inequality,
    \begin{align*}
        \bigg|\sum_{j : i \in j} \frac{a(j)}{\hat{a}(i)}k^Te_{j^C}\bigg|^2
        ~ \le ~
        \sum_{j : i \in j} \frac{a(j)}{\hat{a}(i)}\big|k^Te_{j^C}\big|^2,
        \quad
        \bigg|\sum_{j : i \notin j} \frac{a(j)}{\hat{a}(-i)}k^Te_{j}\bigg|^2
        ~ \le ~
        \sum_{j : i \notin j} \frac{a(j)}{\hat{a}(-i)}\big|k^Te_{j}\big|^2.
    \end{align*}
\end{proof}

\section{Proofs of the results in Section~\ref{sec:appli-1}}

First, we need to develop some results in order to help us show that we can apply the comparison principle to the stochastic filtering problem we discussed above.

\subsection{Estimates on the Sobolev 
norms of products}

We recall the following theorem, the proof of which can be found in \cite{behzadan2021multiplication}.
\begin{theorem}\label{appendix:mul}
    Assume $s_i \in \R$ $(i = 1,2)$ and $s < 0$   are real numbers that satisfy $s_i \ge s$, $\min\{s_1,s_2\} < 0$, $s_1 + s_2 - s > \frac{d}{2}$, $s_1 + s_2 \ge 0$. Then the pointwise multiplication of functions extends uniquely to a continuous linear map
    \begin{align*}
        H^{s_1}(\R^d) \x H^{s_2}(\R^d)
        \longrightarrow H^{s}(\R^d),
    \end{align*}
    i.e.,
    there exists some constant $C > 0$, s.t.
    for any $u \in H^{s_1}(\R^d) $, $v \in H^{s_2}(\R^d) $, 
    \begin{equation*}
        |uv|_{s} \le C|u|_{s_1}|v|_{s_2}.
    \end{equation*}
\end{theorem}

\subsection{Commutator estimates}

The class of Schwartz functions  $\mathcal{S}(\mathbb R^d)$ is the space of smooth functions whose derivatives are bounded by $C_N(1+|\xi|^2)^{-N}$ for every $N \in \mathbb Z^+$, equipped with the topology induced by the family of seminorms
\begin{align*}
    \rho_{\alpha,\beta}(\phi)=\sup_{\xi \in \mathbb R^d}  \left|\xi^{\alpha} \partial_{\beta}\phi(\xi) \right|, \quad \forall \, \phi \in \mathcal{S}({\mathbb R^d}),
\end{align*}
indexed by all multi-indices $\alpha,\beta$. Denote by $\mathcal{S}'(\mathbb R^d)$ the topological dual of $\mathcal{S}(\mathbb R^d)$. 

\begin{definition}\label{def:Sobolev}
    Let $s$ be a real number. The space $H^s(\mathbb R^d)$ is defined as the set of all distributions $u$ in $\mathcal{S}'(\mathbb R^d)$ with the property that 
  x  \begin{align}
        \mathcal{J}_{-s} (u):= (I_d - \Delta)^{s/2} u = \mathcal{F}^{-1} ((1+|\xi|^2)^{s/2} \mathcal{F}u(\cdot) )
    \end{align}
    is an element of $L^2(\mathbb R^d)$. Here, $\mathcal{F}$ denotes the Fourier transform. $\Jc_{-s}$ is called the Bessel potential operator, and $|u|_s:=|\Jc_{-s} u |_{L^2}=|(1+|\xi|^2)^{s/2} \mathcal{F}u(\cdot)|_{L^2} $.
\end{definition}

We collect some basic properties of $H^s(\mathbb R^d)$ and $\mathcal{J}_s$ in the following lemma; see \cite{AdHe96,Gr14} for more details. 
\begin{lemma}
\begin{itemize}
\item For any real numbers $s \leq r$, $H^s(\mathbb R^d) \supset H^r(\mathbb R^d) \supset \mathcal{S}(\mathbb R^d)$. Any finite signed measure $\eta$ is an element of $H^{s}(\mathbb R^d)$ with $s < -d/2$. 
\item For any real number $s,r$,  $\Jc_{-r} \Jc_{-s} f = \Jc_{-(r+s)} f$, for any $f \in H^{r+s}(\R^d) \cap  H^{s}(\R^d)$.
\item If $s$ is a nonnegative integer, $H^{s}(\mathbb R^d)$ is exactly the Sobolev space of order $s$. 
\item For any real number $s,r$, $\Jc_{-s}$ is an isomorphism from $H^{s+r}(\mathbb R^d)$ to $H^r(\mathbb R^d)$. 
\item The Bessel potential operator commutes with derivatives, that is, $D_{x_i} \Jc_{-s} f = \Jc_{-s} D_{x_i} f$, for any $f \in  H^{1+s}(\mathbb R^d)$ and $i=1,\dotso,d$. 
\item For any $f \in H^s(\mathbb R^d) \cap H^{s+r}(\mathbb R^d), g \in H^r(\mathbb R^d)$, it holds that $\langle \Jc_{-s} f, \Jc_{-r} g \rangle_{L^2}= \langle \Jc_{-(s+r)} f, g \rangle_{L^2}$.  
\end{itemize}
\end{lemma}

In the rest, we denote by $\alpha = (\alpha_1, \dots, \alpha_d) \in \mathbb{N}^d$ a multi-index, and write $|\alpha| = \alpha_1 + \cdots + \alpha_d$.
For a smooth function $f : \mathbb{R}^d \to \mathbb{R}$, we denote
\[
D^\alpha_x f = \frac{\partial^{|\alpha|} f}{\partial x_1^{\alpha_1} \cdots \partial x_d^{\alpha_d}}.
\]

\begin{lemma}\label{lem:indu}
    For any integer $k\geq 1$ and any two functions $f,h \in C^{\infty}(\R^d)$, the following holds
    \begin{align}\label{eq:indu}
        \Jc_{-2k}(fh)=f\Jc_{-2k}h+\sum_{ |\alpha| \leq 2k-1}L_{k,\alpha}(f)D_x^{\alpha} h
    \end{align}
    where $L_{k,\alpha}(f)(x)$ is a linear combination of $D^{\alpha}_x f(x)$, with $|\alpha| \le 2k-1$.
\end{lemma}
\begin{proof}
    The proof is by induction. For $k=1$, we have 
\begin{align*}
    (I-\Delta)(fh)=fh-\Delta (fh)=f(I-\Delta)h-2\sum_{i=1}^d D_{x_i} f D_{x_i}h -\Delta f h.
\end{align*}
For the induction step, we use the same inequality on \eqref{eq:indu} to obtain 
\begin{align*}
        \Jc_{-2(k+1)}(fh)=& (I-\Delta)(f\Jc_{-2k}h)+\sum_{|\alpha| \leq 2k-1} (I-\Delta)(L_{k,\alpha}(f)D_x^\alpha h)\\
        =&f\Jc_{-2(k+1)}(h)-2\sum_{i=1}^d D_{x_i} f\Jc_{-2k}D_{x_i}h-\Delta f\Jc_{-2k}h \\
        &+\sum_{|\alpha|\leq 2k-1}L_{k,\alpha}(f) (I-\Delta)D_x^\alpha h-2\sum_{i=1}^d(L_{k,\alpha}(D_{x_i}f)D_x^{\alpha}D_{x_i} h)-L_{k,\alpha}(\Delta f) D_x^\alpha h,
\end{align*}
which allows us to identify $L_{k+1,\alpha}$ for $|\alpha| \leq 2(k+1)-1$.
\end{proof}

In a similar spirit to \cite{gozzi2000hamilton}, we need the following commutator estimate.

\begin{lemma}\label{lem:commutator}
For any integer $k\geq 1$, and any functions $f \in H^{2k + \frac{d}{2} +1}(\R^d) \cap C^{\infty}(\R^d), g \in H^{1-2k}(\R^d)\cap C^{\infty}(\R^d)$, we have 
$$|\Jc_{2k} (fg)-f \Jc_{2k}g|_{L^2}\leq C|f|^2_{2k+d/2+1}|g|_{-2k-1/2}^2,$$
where $C$ is a constant depending on $d,k$.
\end{lemma}

\begin{proof}
By definition and \eqref{eq:indu}, we have 
\begin{align*}
    &~\quad|\Jc_{2k} (fg)-f\Jc_{2k}g|_{L^2}^2
    =|fg-\Jc_{-2k}(f\Jc_{2k}g)|_{-2k}^2=\left|\sum_{|\alpha|\le 2k -1}L_{k,\alpha}(f)D_x^\alpha \Jc_{2k}g \right|_{-2k}^2\\
    &\leq C_k\sum_{|\alpha|\le 2k -1}|L_{k,\alpha}(f)D_x^\alpha \Jc_{2k}g|_{-2k}^2 
    \leq C_k\sum_{|\alpha|\le 2k -1}|L_{k,\alpha}(f)|^2_{\frac{d}{2}+1+|\alpha|}|D^\alpha_x\Jc_{2k}g|^2_{-\frac{1}{2}-|\alpha|}
    \\ 
    &= C_k\sum_{|\alpha|\le 2k -1}|L_{k,\alpha}(f)|^2_{\frac{d}{2}+1+|\alpha|}|D^j_xg|^2_{-2k-\frac{1}{2}-|\alpha|}
    \leq C_k|f|^2_{2k+\frac{d}{2}+1}|g|^2_{-2k-\frac{1}{2}}. 
\end{align*}
where we used Theorem~\ref{appendix:mul} with $s_1 = \frac{d}{2}+1+|\alpha|$, $s_2 = -\frac{1}{2} - |\alpha|$, $s = -2k$ to obtain the second inequality. 

\end{proof}

\begin{remark}
    The commutator estimate above is inspired by \cite{gozzi2000hamilton}, in which they consider an HJB equation for the optimal control of Duncan-Mortensen-Zakai (DMZ) equation. 
    For relaxed solutions to the controlled  DMZ equation, they need the corresponding commutator estimate to generalize the standard estimates from the Sobolev norm to the weighted Sobolev norm.
    In our work, we use the commutator estimates to control the $D_\mu u$ terms that arise in the argument of doubling the variable, which are the problematic terms in terms of convergence.
    The idea is to use the difference terms in the other derivative $D_{x\mu}\mu$ to control the problematic terms.
    When $\sigma$ is state independent, this control is easy. We need to use the commutator estimates for state-dependent diffusion coefficients $\sigma$.
\end{remark}

For given functions $a,b$, define the operators
\begin{align*}
\mathcal{B}f(x)&=b^\top(x)D_xf(x),\,\mathcal{A}f(x)=\frac{1}{2}Tr\left(a(x)D^2_x f(x)\right).\\
\end{align*}

\begin{proposition}\label{prop:commutator}
Assume that $\xi^{\top} a(x)\xi \geq \delta |\xi|^2$ for all $x,\xi \in \R^d$.
    For any finite signed measure $\eta$, we have 
    \begin{align}\label{eq:boundf}
        \int_{\R^d} \left( \Ac(\Jc_{2\lambda}\eta)(x) + \Bc(\Jc_{2\lambda}\eta)(x)\right)\eta(dx) \leq -\frac{\d}{4} |\eta|_{1-\lambda}^2+c|\eta|_{-\lambda}^2
    \end{align}
    for a strictly positive constant $c$ depending on $\sup_{i,j}|a_{i,j}|_{k+d/2+1}.$
\end{proposition}
\begin{proof} 
We first show that for any smooth function $f \in H^{2-\lambda}(\mathbb R^d) \cap C^{\infty}(\mathbb R^d)$, we have 
    \begin{align}\label{eq:boundg}
        \int_{\R^d} \left( \Ac(\Jc_{2\lambda}f)(x) + \Bc(\Jc_{2\lambda}f)(x)\right)f(x)\,dx \leq -\frac{\d}{4} |f|_{1-\lambda}^2+c|f|_{-\lambda}^2,
    \end{align}
then we extend this result to signed measures.

    We first estimate the $\Ac$ term as follows
\begin{align*}
        &\quad~\int_{\R^d} \Ac(\Jc_{2\lambda}f)(x) f(x)\,dx= \frac{1}{2} \int Tr\left(a(x)D^2_x(\Jc_{2\lambda}f)(x)\right) f(x)\,dx\\
        &= \frac{1}{2} \sum_{i,j=1}^d\int_{\R^d} ( \Jc_{2\lambda}D_{x_ix_j}^2f)(x) (a_{i,j}f)(x)\,dx
            =-\frac{1}{2} \sum_{i,j=1}^d\int_{\R^d} \big(\Jc_{\lambda}D_{x_j}f\big)(x) \big(\Jc_{\lambda} D_{x_i}(a_{i,j}f)\big)(x)\,dx\\
        &=-\frac{1}{2} \sum_{i,j=1}^d\int_{\R^d} \big(\Jc_{\lambda}D_{x_j}f\big)(x) \left(\Jc_{\lambda}(a_{i,j}D_{x_i}f)(x)+\Jc_{\lambda}(fD_{x_i}a_{i,j})(x)\right)\,dx\\
        &=-\frac{1}{2} \sum_{i,j=1}^d\int_{\R^d} \big(\Jc_{\lambda}D_{x_j}f\big)(x)a_{i,j}(x)\big(\Jc_{\lambda}D_{x_i}f\big)(x)\,dx\\
            &\quad-\frac{1}{2} \sum_{i,j=1}^d\int_{\R^d} \big(\Jc_{\lambda}D_{x_j}f\big)(x) \left(\Jc_{\lambda}(a_{i,j}D_{x_i}f)(x)
            -a_{i,j}(x)\big(\Jc_{\lambda}D_{x_i}f\big)(x)
            +\Jc_{\lambda}(fD_{x_i}a_{i,j})(x)\right)\,dx\\
        &\leq -\frac{\d}{2}\int_{\R^d} \Big|\big(\Jc_{\lambda}D_{x_j}f\big)(x)\Big|^2\,dx+C_{\lambda,d}\sup_{i,j}|a_{i,j}|_{\lambda+d/2+1}| f|_{1-\lambda}\left(|f|_{1/2-\lambda}+|f|_{-\lambda}\right)\\
        &\leq -\frac{\d}{2}| f|^2_{1-\lambda}+\tilde C_{\lambda,d}|f|_{1-\lambda}|f|_{1/2-\lambda} \leq  -\frac{\d}{2}| f|^2_{1-\lambda}+\tilde C_{\lambda,d}|f|_{1-\lambda}^{3/2}\sqrt{|f|_{-\lambda}},
\end{align*}
where we used Lemma \ref{lem:commutator} with $2k = \lambda$, Theorem \ref{appendix:mul} with $s = -\lambda$, $s_1 = -\lambda$, $s_2 = \lambda + \frac{d}{2} + 1$ to obtain the first inequality,
and interpolation inequality to obtain the last inequality. 
Note that for any $A,B,\geq 0,$ $C>0$ we have 
$$A^{3/2}B^{1/2}=\left(\frac{A}{C}\right)^{3/2}(C^3B)^{1/2}\leq \frac{3}{4C^2}{A^2}+ \frac{C^6}{4}B^2.$$
Thus, up to changing $\tilde C_{\lambda,d}$ to $\tilde C_{\lambda,d,\d}$, we have
\begin{align*}
        \int_{\R^d} \Ac(\Jc_{2\lambda}f)(x) f(x)\,dx \leq  -\frac{\d}{3}| f|^2_{1-\lambda}+\tilde C_{\lambda,d,\d}|f|^2_{-\lambda}.
\end{align*}
We now estimate
\begin{align*}
      &\int_{\R^d} \Bc(\Jc_{2\lambda}f)(x)f(x)dx
      \\=&\int_{\R^d}(\Jc_{2\lambda}D_xf)(x) b(x)f(x)dx=\int_{\R^d} (\Jc_{\lambda}D_xf)(x) \Jc_{\lambda}(bf)(x)dx \\
      \le &~ |f|_{1 - \lambda}|bf|_{-\lambda}
      \le C|b|_{\lambda}|f|_{1 - \lambda}|f|_{-\lambda}.
\end{align*}
Thus, using $\epsilon-$Young and Holder inequality, we obtain \eqref{eq:boundg}.

Since $2 - \lambda < -\frac{d}{2}$, we have that
$\eta \in H^{2 - \lambda}(\R^d)$, 
define  $\eta_{\e}= \eta \ast G_{\e}$, and claim that $\eta_\e \in H^{2-\lambda} \cap C^\infty(\R^d)$, where $G_{\e}=\frac{1}{\sqrt{2 \pi \e^2 }}e^{-\frac{|x|^2}{2\e^2}}$ is the density function of Gaussian distribution with mean $0$ and variance matrix $\e^2 I_d$. 

It is clear that $\eta_\e \in C^\infty(\R^d)$.
Note that $|\Fc (\eta)|_{L^{\infty}}$ is bounded by the total variation $M$ of $\eta$, and $\Fc(G_\e)(\xi)=\exp(-2\pi^2 \e^2 |\xi|^2)$, hence 
\begin{align*}
    |\eta_{\e}|_{2-\lambda}=& | (1+|\xi|^2)^{(2-\lambda)/2} \Fc( \eta) \Fc(G_\e))|_{L^2}  \\
    \leq& M \left|  (1+|\xi|^2)^{(2-\lambda)/2}\exp(-2\pi^2 \e^2 \xi^2) \right|_{L^2} < +\infty,
\end{align*}
i.e. $\eta_\e \in H^{2-\lambda}(\R^d)$.

Moreover, we claim that 
for each $i, j = 1,\cdots,d$,
\begin{align*}
 D_{x_ix_j}^2 \Jc_{\lambda} \eta_{\e} \longrightarrow D_{x_ix_j}^2 \Jc_{\lambda}\eta, \quad \text{ as $\e \longrightarrow 0$} 
\end{align*}
in $L^\infty$-norm, and $\eta_{\e} \to \eta$ in the topology of weak convergence. 
Indeed, $\Jc_{\lambda} \eta_{\e} \in H^2(\mathbb R^d) \cap C^2(\mathbb R^d)$, and due to the explicit formula 
\begin{align*}
D_{x_ix_j}^2 \Jc_{\lambda} \eta_{\e}(x)= - (2\pi)^{-d/2}\int e^{i x \cdot \xi}\xi_i\xi_j   \frac{\Fc( \eta)(\xi)\Fc( G_{\e})(\xi)}{(1+|\xi|^2)^{\lambda/2}} \, d\xi,
\end{align*}
we have that 
\begin{align*}
  \left|D_{x_ix_j}^2 \Jc_{\lambda} \eta_{\e}(x)- D_{x_ix_j}^2 \Jc_{\lambda} \eta(x) \right|   \leq  M \left|\int  \frac{|\Fc( G_{\e})(\xi)-1|}{(1+|\xi|^2)^{\lambda/2-1}} \, d\xi \right| \longrightarrow 0. 
\end{align*}
Noting that for any $h \in C_b(\R^d)$, we have that
\begin{align*}
    \int_{R^d}h(x)(\eta_\e(x)dx - \eta(dx))
    ~ = &~
    \int_{\R^d} h(x)\int_{\R^d}G_\e(x-y)\,\eta(dy)\,dx -  \int_{R^d}h(x)\,\eta(dx)
    \\~ = &~
    \int_{\R^d}\int_{\R^d}h(x)G_\e(x-y)\,dx\,\eta(dy) - \int_{R^d}h(y)\,\eta(dy)
    \\~ = &~
    \int_{\R^d} \big(h*G_\e(y) - h(y)\big)\, \eta(dy)
    \longrightarrow 0,
    ~\mbox{as}~
    \e \longrightarrow 0,
\end{align*}
by dominated convergence theorem.
Thus, we have the following estimation 
\begin{align*}
    & ~ 
    \bigg|\int_{\R^d} \Ac(\Jc_{\lambda} \eta_{\e})(x) \eta_\e(x)dx - \int_{\R^d} \Ac (\Jc_{\lambda} \eta)(x) \eta(dx)\bigg|
    \\ ~ \le & ~
    \bigg|\int_{\R^d} \big(\Ac(\Jc_{\lambda} \eta_{\e})(x)  - \Ac (\Jc_{\lambda} \eta)(x)\big) \eta_\e(x)dx\bigg|
    +
    \bigg|\int_{\R^d} \Ac(\Jc_{\lambda} \eta)(x) \big(\eta_\e(x)dx- \eta(dx)\big)\bigg|
    \\ ~ \le & ~
    \int_{\R^d} \frac{1}{2}|Tr\big(a(x)\big((D^2_x\Jc_{\lambda} \eta_{\e})(x)  - (D^2_x\Jc_{\lambda} \eta)(x)\big)\big)||\eta_\e(x)|dx
    +
    \bigg|\int_{\R^d} \Ac(\Jc_{\lambda} \eta)(x) \big(\eta_\e(x)dx- \eta(dx)\big)\bigg|
    \\ ~ \le & ~
    \frac{M}{2}\sum_{i,j = 1}^d |a_{j,i}(x)|\Big|D_{x_ix_j}^2 \Jc_{\lambda} \eta_{\e}(x) \Big|
    +
    \bigg|\int_{\R^d} \Ac(\Jc_{\lambda} \eta)(x) \big(\eta_\e(x)dx- \eta(dx)\big)\bigg|
    \\ ~ \le & ~
    \frac{M^2d^2}{2}\sup_{i,j}|a_{i,j}|_{L^{\infty}}  \left|\int  \frac{|\Fc (G_{\e})(\xi)-1|}{(1+|\xi|^2)^{\lambda/2-1}} \, d\xi \right|
    + \bigg|\int_{\R^d} \Ac(\Jc_{\lambda} \eta)(x) \big(\eta_\e(x)dx- \eta(dx)\big)\bigg|
    \longrightarrow 0, 
\end{align*}
as $\e \longrightarrow 0$.
Similarly, we have
\begin{align*}
    \bigg|\int_{\R^d} \Bc(\Jc_{\lambda} \eta_{\e})(x) \eta_\e(x)dx - \int_{\R^d} \Bc (\Jc_{\lambda} \eta)(x) \eta(dx)\bigg|
    \longrightarrow 0, 
    ~\mbox{as}~
    \e \longrightarrow 0.
\end{align*}

\end{proof}

\subsection{Proof of Theorem~\ref{thm:viscosity_property}}\label{Proof:thm:viscosity}

Assuming that $b,\sigma,\tilde \sigma, r,l$ are bounded and Lipschitz, we can prove that $(t,\mu) \mapsto v(t,\mu)$ is continuous as in \cite{BaEkZh23}. 
By an argument parallel to that of Proposition 6.3 in \cite{DaJaSe23}, the value function is a viscosity solution to \eqref{eq:filtering}. 

For uniqueness, we verify that the Hamiltonian 
$$G^e: \Pc_2(\mathbb R^d) \x \mathbb R^d \times B_l^d \x B_l^{d \x d} \x \S_d \longrightarrow \mathbb R; \, (\mu,m,p,q,M) \mapsto \inf_{a \in A} K^e(a,\mu,m,p,q,M)$$ satisfies Assumption~\ref{ass:comparison},
where $K^e(a,\mu,m,p,q,M):= K(a,(I_d + m)_\sharp\mu,p(\cdot-m),q(\cdot - m),M).$

Due to the inequality 
\begin{align*}
  \left| G^e(\mu,m,p_1,q_1,M_1)- G^e(\mu,m,p_2,q_2,M_2)  \right| \leq \sup_{a \in A} \left| K^e(a, \mu,m, p_1,q_1,M_1) - K^e(a,\mu,m,P_2,q_2,M_2) \right|, 
\end{align*}
for Assumption~\ref{ass:comparison} (i), it suffices to show that for any fixed $a$
\begin{align}\label{eq:filtering1}
   & \left| K^e(a, \mu,m, p_1,q_1,M_1) - K^e(a,\mu,m,p_2,q_2,M_2) \right|  \\
    & \leq L_G \bigg(1 + |m| +
            \int_{\R^d}|x| \, \mu(dx)\bigg)
             \times \Big(|p_1 - p_2|_l + |q_1 - q_2|_l  + \|M_1 - M_2\| \Big), \notag
\end{align}
where $L_G$ is some positive constant independent of $a$. 
Indeed from the construction of $K$, it is straightforward that 
\begin{align*}
    & \left| K(a, (I_d+m)_{\#}\mu, p_1(\cdot -m),q_1(\cdot -m),M_1) - K(a,(I_d+m)_{\#}\mu,p_2(\cdot -m),q_2(\cdot - m),M_2) \right|  \\
    & \leq  C \left( \int_{\R^d} |p_1(x)-p_2(x)| + |q_1(x)-q_2(x)| \, \mu(dx) + ||M_1-M_2|| \right)  \\
    & \leq C \left( \bigg(1 + 
            \int_{\R^d}|x| \, \mu(dx)\bigg)
             \times \Big(|p_1 - p_2|_l + |q_1 - q_2|_l \Big)+||M_1-M_2||\right),
\end{align*}
and hence verifies \eqref{eq:filtering1}.

For Assumption \eqref{ass:comparison} (ii), recalling and defining that 
\begin{align*}
      \kappa(x)
      ~ = ~
      \frac{1}{\e}\int_{\R^d}
            \frac{Re(F_k(\mu - \nu)f^*_k(x))}
            {(1 + |k|^2)^\lambda}
            dk,
    \quad
    \mu_m := (I_d + m)_\sharp\mu,
\end{align*}
and $M \mapsto K(a,\mu,p,q,M)$ is increasing,  it is enough to estimate 
\begin{align*}
    & K(a, \mu_m, \nabla \kappa(\cdot -m), \nabla^2 \kappa (\cdot -m), M) - K(a,\nu_n, \nabla \kappa (\cdot -n), \nabla^2 \kappa (\cdot -n), M)  \\
     &\leq K(a, \mu_m, \nabla \kappa(\cdot -m), \nabla^2 \kappa (\cdot -m), M) - K(a,\nu_m, \nabla \kappa (\cdot -m), \nabla^2 \kappa (\cdot -m), M) \\
    & \ \ \ + K(a,\nu_m, \nabla \kappa (\cdot -m), \nabla^2 \kappa (\cdot -m), M) - K(a,\nu_n, \nabla \kappa (\cdot -n), \nabla^2 \kappa (\cdot -n), M).
    \end{align*}
  According to the definition of $K$, the first term on the right hand side is bounded by 
    \begin{equation}\label{eq:final1}
        \begin{split}
        &  \int_{\R^d} r(x+m,\mu_m, a)\,\mu(dx) - \int_{\R^d} r(x+m,\nu_m,a)\,\nu(dx)
        \\ &+\int_{\R^d} b(x+m,\mu_m, a)^\top \nabla \kappa(x)\, \mu(dx) -
            \int_{\R^d} b(x+m,\nu_m, a)^\top \nabla \kappa(x)\, \nu(dx) 
        \\
        &+\frac{1}{2} \int_{\R^d}Tr\left(\nabla^2 \kappa(x) \sigma \sigma^\top (x+m,a) \right)\, (\mu - \nu)(dx)  \\
         \le&  \int_{\R^d} \big(r(x+m,\mu_m, a) -r(x+m,\nu_m,a)\big)\,\mu(dx) + \int_{\R^d} r(x+m,\nu_m,a)\,(\mu - \nu)(dx)
        \\ &+\int_{\R^d} \big(b(x+m,\mu_m, a) - b(x+m,\nu_m, a)
        \big)^\top \nabla \kappa(x)\, \mu(dx) 
        \\& + \int_{\R^d} b(x+m,\nu_m, a)^\top \nabla \kappa(x) (\mu - \nu)(dx) 
            + \frac{1}{2} \int_{\R^d}\text{Tr}\left(\nabla^2 \kappa(x) \sigma \sigma^\top (x+m,\nu_m,a) \right) (\mu - \nu)(dx)
        \\
        &+\frac{1}{2} \int_{\R^d}\text{Tr}\Big(\nabla^2 \kappa(x) 
        \big(\sigma \sigma^\top (x+m,\mu_m,a) - \sigma \sigma^\top (x+m,\nu_m,a)\big)\Big)\, \mu(dx) \\
         \leq&~ C |\mu_m -\nu_m|_{-\lambda} 
         + L|\mu_m -\nu_m|_{-\lambda}(|\nabla\kappa|_{L^\infty} + |\nabla^2\kappa|_{L^\infty})
         -\frac{\d}{4\e} |\mu-\nu|_{1-\lambda}^2+\frac{C}{\e}|\mu-\nu|_{-\lambda}^2, 
        \end{split}        
    \end{equation}
where we apply Proposition~\ref{prop:commutator} with $\eta=\mu-\nu$ together with the equality $\kappa= \frac{1}{\e} \mathcal{J}_{2\lambda} \eta$. 
By direct computation, we get that
$$
    |\mu_m - \nu_m|^2_{-\lambda} 
    =
    \int_{\mathbb R^d} \frac{|F_{k}(\mu-\nu)e^{ik\cdot m}|^2}{(1+|k|^2)^{\lambda}} \, dk
    =
    \int_{\mathbb R^d} \frac{|F_{k}(\mu-\nu)|^2}{(1+|k|^2)^{\lambda}} \, dk
    =
    |\mu - \nu|^2_{-\lambda}.
$$

Using the Lipschitz property of $r,b, \sigma\sigma^\top$, the second term is bounded by 
\begin{align*}
     & \int_{\R^d} \big(r(x+m,\nu_m,a)-r(x+n,\nu_n,a)\big) \, \nu(dx)
     \\ & +\int_{\R^d} (b(x+m,\nu_m,a)^\top-b(x+n,\nu_n,a)^\top) \nabla \kappa(x) \, \nu(dx)
     \\ & + \int \frac{1}{2} Tr\left(\nabla^2 \kappa(x)  (\sigma \sigma^\top (x+m,\nu_m,a) -\sigma \sigma^\top (x+n,\nu_n,a))\right)\,  \nu(dx)
     \\ \le &~ C(|m-n| + |\nu_m - \nu_n|_{-\lambda}) + C(|m-n|+|\nu_m - \nu_n|_{-\lambda})(|\nabla\kappa|_{L^{\infty}}+|\nabla^2\kappa|_{L^{\infty} }).
\end{align*}
By direct computation, we get that 
\begin{align*}
    |\nu_m - \nu_n|^2_{-\lambda}
    = &
    \int_{\mathbb R^d} \frac{(2\pi)^{-d/2}|\langle e^{ik\cdot(x+m)} - e^{ik\cdot (x+n)},\nu\rangle|^2}{(1+|k|^2)^{\lambda}} \, dk
    \\
    \le &
    \int_{\mathbb R^d} \frac{(2\pi)^{-d/2}|k|^2|m-n|^2}{(1+|k|^2)^{\lambda}} \, dk
    \le C|m - n|^2,
    \\
    |\nabla \kappa (x) | \leq& \frac{1}{\e} \int_{\mathbb R^d} \frac{|k F_k(\mu-\nu)|}{(1+|k|^2)^{\lambda}} \, dk \leq \frac{\rho_F(\mu,\nu)}{\e} \sqrt{\int_{\mathbb R^d} \frac{|k|^2}{(1+|k|^2)^{\lambda}} \,dk}=\frac{C\rho_F(\mu,\nu)}{\e}, 
\end{align*}
and similarly $ |\nabla^2 \kappa (x) |\leq \frac{C\rho_F(\mu,\nu)}{\e}$. Therefore, we get the bound $L\left(d_F(\theta,\iota)+\frac{d_F(\theta,\iota))^2}{\e} \right)$ where $\theta=(\mu,m)$ and $\iota=(\nu,n)$. Together with \eqref{eq:final1}, this verifies Assumption~\ref{assumption1} (ii).

\endproof

\bibliography{references}
\bibliographystyle{plain}

\end{document}